\documentclass[a4paper,11pt]{scrartcl}
\usepackage[utf8]{inputenc}
\usepackage{amssymb,amsfonts,amsmath,amsthm,graphicx}
\usepackage{url,tikz,float}
\usetikzlibrary{matrix}
\usetikzlibrary{arrows,decorations.pathmorphing}

\newtheorem{theorem}{Theorem}
\newtheorem{lemma}[theorem]{Lemma}
\newtheorem{proposition}[theorem]{Proposition}
\newtheorem{corollary}[theorem]{Corollary}
\newtheorem{definition}[theorem]{Definition}
\newtheorem{conjecture}[theorem]{Conjecture}
\newtheorem{remark}[theorem]{Remark}
\newcommand{\vect}[1]{\boldsymbol{#1}}
\newcommand{\first}{\mathsf{First}}
\newcommand{\last}{\mathsf{Last}}
\newcommand{\even}{\mathsf{Even}}
\renewcommand{\int}{\mathsf{int}}
\newcommand{\mF}{\mathcal{F}}
\newcommand{\core}{\mathsf{core}}

\title{New reducible configurations for graph multicoloring with application to the experimental resolution of McDiarmid-Reed's Conjecture\\ (extended version)}
\author{Jean-Christophe Godin $^a$ \and Olivier Togni $^b$}

\begin{document}

\maketitle
\begin{center}
$^a$ Institut de Math\'ematiques de Toulon, Universit\'e de Toulon, France\\
\texttt{godinjeanchri@yahoo.fr}
\medskip

$^b$  Laboratoire LIB, Université de Bourgogne, France\\
\texttt{olivier.togni@u-bourgogne.fr}
\medskip
\end{center}

\begin{abstract}
A {\em $(a,b)$-coloring} of a graph $G$ associates to each vertex a $b$-subset of a set of $a$ colors in such a way that the color-sets of adjacent vertices are disjoint.
We define general reduction tools for $(a,b)$-coloring of graphs for $2\le a/b\le 3$. 
In particular, using necessary and sufficient conditions for the existence of a $(a,b)$-coloring of a path with prescribed color-sets on its end-vertices, more complex $(a,b)$-colorability reductions are presented. The utility of these tools is exemplified on finite triangle-free induced subgraphs of the triangular lattice for which McDiarmid-Reed's conjecture asserts that they are all $(9,4)$-colorable. Computations on millions of such graphs generated randomly show that our tools allow to find a $(9,4)$-coloring for each of them except for one specific regular shape of graphs (that can be $(9,4)$-colored by an easy ad-hoc process). We thus obtain computational evidence towards the conjecture of McDiarmid\&Reed.
\end{abstract}

\section{Introduction}
For two integers $a$ and $b$, a {\em $(a,b)$-coloring} of a graph $G$ is a mapping that associates to each vertex a set of $b$ colors from a set of $a$ colors in such a way that adjacent vertices get disjoints sets of colors. In particular, a $(a,1)$-coloring is simply a proper coloring.
Equivalently, a $(a,b)$-coloring of $G$ is a homomorphism to the Kneser graph $K_{a,b}$. This type of coloring is also in relation with fractional colorings: the fractional chromatic number of a graph $G$ can be defined as $\chi_f (G)=\min\{a/b, G \text{ is } (a,b)\text{-colorable}\}$.

Since the introductory paper of Stahl~\cite{sta76}, multicoloring attracted much research, one of the most recent results being the one of Cranston and Rabern~\cite{CR18} showing that planar graphs are $(9,2)$-colorable (without using the four color theorem).

An associated problem is weighted coloring, in which the number of colors to assign can vary from vertex to vertex. Weighted coloring has a natural application for frequency allocation in cellular networks~\cite{jan00}. In particular, since the equidistant placement of transmitters induces a triangular lattice, weighted coloring of triangular lattices has been the subject of many studies~\cite{jan00,mcd00,nara01,spa12}. In most of these works, the aim is to bound the weighted chromatic number by a constant times the weighted clique number. Multicoloring is a particular case of weighted coloring; however, the fact that a graph is $(a,b)$-colorable can be translated into a bound for weighted chromatic number in terms of weighted clique, see \cite{wit} for instance.

%In this paper, we concentrate our attention on $(a,b)$-coloring with $2\le a/b\le 3$, hence on graphs of maximum clique size at most 3.
Any graph with at least one edge is not $(a,b)$-colorable for $a < 2b$, so it is useful to introduce $e=a-2b$, $e$ symbolizing the entropy of the $(a,b)$-coloring. Note also that a graph is $(2b,b)$-colorable if and only if it is bipartite. In this paper we concentrate on the pairs $(a,b)$ such that $2 < \frac{a}{b} \leq 3 $, thus $e \in \{1,\ldots,b\}$. It is easy to observe that if a graph is $(a,b)$-colorable, then it is also $(am,bm)$-colorable for any $m\ge 1$. Moreover, the following decomposition property also holds: for any integer $y \geq 1$, if $G$ is $(2x+1,x)$-colorable for any $x \in \{1,\ldots,y\}$, then $G$ is $(a,b)$-colorable for any $a,b$ such that $\frac{a}{b}\geq\frac{2y+1}{y}$. Hence $(2b+1,b)$-colorings are the extremal objectives for non-bipartite graphs.

For finite triangle-free induced subgraphs of the triangular lattice, called {\em hexagonal graphs} in this paper, it is easy to observe that they are $(3,1)$-colorable, and it has been proven by Havet~\cite{hav01} that such graphs are also $(5,2)$-colorable and $(7,3)$-colorable. Sudeep and Vishwanathan~\cite{Sud05} then presented a simpler  $(14,6)$-coloring algorithm and later, Sau et al.~\cite{Sau12} a simpler $(7,3)$-coloring, but using the four colors theorem as a subroutine. The list version has been considered by Aubry et al.~\cite{AGT10}. They showed that hexagonal graphs are $(5m,2m)$-choosable. Of course, as a hexagonal graph can have an induced 9-cycle, the best we can hope is to find a $(9,4)$-coloring. In 1999, McDiarmid and Reed proposed the following conjecture (initially stated in terms of weighted coloring):
\begin{conjecture}[McDiarmid-Reed~\cite{mcd00}]\label{mcdconj}
Every hexagonal graph is $(9,4)$-colorable.
\end{conjecture}

This paper aims to define general reducible configurations for $(a,b)$-coloring when $2< a/b\le 3$, which may allow one to prove that some graphs are $(a,b)$-colorable. In particular, we will apply our reduction tools on triangle-free induced subgraphs of the triangular lattice, in order to solve Conjecture~\ref{mcdconj}.
A part of the results is based on the Ph.D. thesis of Godin~\cite{god09} that also contains reducibility results for $(a,b)$-choosability, even when $a/b\ge 3$.
Searching for reducible configurations is common when dealing with graph (multi)coloring. It is, for instance, a part of the discharging method, extensively used for proving colorability properties on planar or bounded maximum average degree graphs.
In the case $2< a/b\le 3$, reducible configurations for $(a,b)$-coloring take the form of induced paths in which interior vertices have degree 2 in the graph, called {\em handles}. In order to get sharper results, we have also to define more sophisticated handles by looking at the lengths of the induced paths that start at one or both of the end-vertices of the handle (in which cases we speak about S-handle and H-handle, respectively).

The paper is organized as follows: In Section 2, we define the different types of handles we will consider and present the general handle-reducibility results. As the associated proofs are technical and quite long, we postpone them to the last sections of the paper. In Section 3, we apply our reduction tools on finite triangle-free induced subgraphs of the triangular lattice, in order to try to solve Conjecture~\ref{mcdconj}. We also present the results of our computations, giving empirical evidence towards the conjecture and listing some possible extensions of this work and a conjecture generalizing Conjecture~\ref{mcdconj}. In Section 4, the necessary and sufficient conditions under which a path is $(a,b)$-colorable when its end-vertices are already precolored are determined. These results constitute the basic tools we will use to prove the results presented in Section 2. Section 5 is devoted to finding special $(2b+1,b)$-colorings of a path in order to prove handle-reducibility in the next section. In Section 6, we prove the reducibility results for S-handles and H-handles.

\section{Handle reductions}
\label{sec2}
The path $P_{n+1}$ of length $n$ is the graph with vertex set $\{0,\ldots,n\}$ and edge set $\{i(i+1), i=0,1,\ldots,n-1\}$.

A {\em handle} $P(n)$ of length $n$ in a graph $G$ is a path of length $n$ that is an induced subgraph of $G$ with vertices of degree two in $P(n)$ having the same degree in $G$. The {\em interior} $\int(H)$ of a handle $H$ is the set of vertices of degree 2 of $H$.

A {\em parity handle (or P-handle)} $PP(n)$ in a graph $G$ is a handle $P(n)$ with the additional property that there exists another path of length $m\le n$, of the same parity than $n$ in $G-\int(P(n))$, between the two end-vertices of $P(n)$. 

An {\em S-handle} $S(n_1,n_2,n_3)$ in a graph $G$ is a handle of length $n_1$ such that one of its end-vertices has degree 3 in $G$ and is also the end-vertex of two other handles of length $n_2$ and $n_3$.

%in $G$ and there is another P-handle in $G$ between an induced subgraph of $G$ that consists in three paths of length $n_1$, $n_2$ and $n_3$ with one end-vertex of each attached together and with vertices of degree two or three in $S(n_1,n_2,n_3)$ having the same degrees in $G$.
An {\em H-handle} $H(n_1,n_2,n,n_3,n_4)$ in a graph $G$ is a handle of length $n$ such that its two end-vertices are of degree three in $G$ and one of them is also the end-vertex of two other handles of length $n_1$ and $n_2$ and the other end-vertex is also the end-vertex of two handles of length $n_3$ and $n_4$. 
Examples of handles are illustrated in Figure~\ref{handles}.

For a handle $H=S(n_1,n_2,n_3)$, the two end-vertices of the paths of length $n_2, n_3$ of $H$ are called the {\em ports} of $H$ and similarly for  $H=H(n_1,n_2,n,n_3,n_4)$, the four end-vertices of the paths of length $n_1,n_2,n,n_3,n_4$ of $H$ are called its ports.
% represent a handle on the top left, P-handle on the top right, S-handle on the bottom left, and H-handle on the bottom right.% and hexagonal handle on the bottom centre. 
%an induced subgraph of $G$ that consists in a path of length $n_1$ and a path of length $n_2$ attached to one extremity of a path of length $n$ and a path of length $n_3$ and a path of length $n_4$ attached to the other extremity and with vertices of degree two or three in $H(n_1,n_2,n,n_3,n_4)$ having the same degrees in $G$.
Due to symmetry reasons, we will consider only S-handles with $n_1\ge n_2 \ge n_3$ and H-handles with $n_1\le n_2 \le n$ and $n_4\le n_3\le n$.
Note that (some of) the ports of a handle may be the same vertices (hence a handle may induce a cycle in the graph).

% for a P-handles or PP-handle, the set of vertices of degree two of the path of length $n_1$ for a handle $S(n_1,n_2,n_3)$ and the set of vertices of degree two of the 'central' path of length $n$ for a handle $H(n_1,n_2,n,n_3,n_4)$.

%In the rest of the paper, we use the generic term {\em handle} to designate a P-handle, PP-handle, S-handle or H-handle.

\begin{figure}[ht]
\begin{center}

\begin{tikzpicture}[scale=0.5]
\node at (0,3) [circle,draw=black,scale=0.8](x1){};
\node at (1,3) [circle,draw=black,fill=black,scale=0.4](x2){};
\node at (2,3) [circle,draw=black,fill=black,scale=0.4](x3){};
\node at (3,3) [circle,draw=black,fill=black,scale=0.4](x4){};
\node at (4,3) [circle,draw=black,fill=black,scale=0.4](x5){};
\node at (5,3) [circle,draw=black,scale=0.8](x6){};
\node at (2.5,2) (){$P(5)$};
%\node at (5,4) [circle,draw=black,fill=black,scale=0.4](y1){};
%\node at (5,5) [circle,draw=black,scale=0.8](y2){$v_4$};
%\node at (5,2) [circle,draw=black,fill=black,scale=0.4](y3){};
%\node at (5,1) [circle,draw=black,fill=black,scale=0.4](y4){};
%\node at (5,0) [circle,draw=black,scale=0.8](y5){$v_3$};
%\node at (0,4) [circle,draw=black,scale=0.8](z1){$v_1$};
%\node at (0,2) [circle,draw=black,fill=black,scale=0.4](z2){};
%\node at (0,1) [circle,draw=black,scale=0.8](z3){$v_2$};
\draw  (x1) -- (x2) -- (x3) -- (x4) -- (x5) -- (x6);
\draw  (5.5,3.5) -- (x6) -- (5.5,2.5);
\draw  (-.5,2.5) -- (x1) -- (-.5,3.5);

%PP(4)
\node at (8,3) [circle,draw=black,scale=0.8](t1){};
\node at (9,3) [circle,draw=black,fill=black,scale=0.4](t2){};
\node at (10,3) [circle,draw=black,fill=black,scale=0.4](t3){};
\node at (11,3) [circle,draw=black,fill=black,scale=0.4](t4){};
\node at (12,3) [circle,draw=black,scale=0.8](t5){};
%\node at (12.8,2.5) [circle,draw=black,fill=black,scale=0.4](s2){};
\node at (10,4.5) [circle,draw=black,fill=black,scale=0.4](s3){};
%\node at (10,5.5) [circle,draw=black,fill=black,scale=0.4](s4){};
\node at (10,2) (){$PP(4)$};
\draw  (7.5,2.5) -- (t1) -- (7.5,3.5);
\draw  (t1) -- (t2) -- (t3) -- (t4) -- (t5) -- (s3) -- (t1);
\draw  (12.5,2.5) -- (t5) -- (12.5,3.5);
\draw (9.5,5) -- (s3) -- (10.5,5);

% S(4,3,2)
\node at (0+14.5,6+-3) [circle,draw=black,scale=0.8](x1){};
\node at (1+14.5,6+-3) [circle,draw=black,fill=black,scale=0.4](x2){};
\node at (2+14.5,6+-3) [circle,draw=black,fill=black,scale=0.4](x3){};
\node at (3+14.5,6+-3) [circle,draw=black,fill=black,scale=0.4](x4){};
%\node at (4+14.5,6+-3) [circle,draw=black,fill=black,scale=0.4](x5){};
\node at (4+14.5,6+-3) [circle,draw=black,scale=0.8](x6){};
\node at (0+14.5,6+-4) [circle,draw=black,fill=black,scale=0.4](y1){};
\node at (0+14.5,6+-5) [circle,draw=black,fill=black,scale=0.4](y2){};
\node at (0+14.5,6+-2) [circle,draw=black,fill=black,scale=0.4](y3){};
\node at (0+14.5,6+-1) [circle,draw=black,fill=black,scale=0.4](y4){};
\node at (0+14.5,6+0) [circle,draw=black,fill=black,scale=0.4](y5){};
\node at (2.5+14.5,6+-4) (){$S(4,3,2)$};
\draw  (x1) -- (x2) -- (x3) -- (x4) -- (x6);
\draw  (y5) -- (y4) -- (y3) -- (x1) -- (y1) -- (y2);
\draw  (-.5+14.5,6+.5) -- (y5) -- (.5+14.5,6+.5);
\draw  (-.5+14.5,6+-5.5) -- (y2) -- (.5+14.5,6+-5.5);
\draw  (4.5+14.5,6+-3.5) -- (x6) -- (4.5+14.5,6+-2.5);

%H(1,2,5,3,2)
\node at (8+13.5,6+-3) [circle,draw=black,scale=0.8](x1){};
\node at (9+13.5,6+-3) [circle,draw=black,fill=black,scale=0.4](x2){};
\node at (10+13.5,6+-3) [circle,draw=black,fill=black,scale=0.4](x3){};
\node at (11+13.5,6+-3) [circle,draw=black,fill=black,scale=0.4](x4){};
\node at (12+13.5,6+-3) [circle,draw=black,fill=black,scale=0.4](x5){};
\node at (13+13.5,6+-3) [circle,draw=black,scale=0.8](x6){};
\node at (13+13.5,6+-4) [circle,draw=black,fill=black,scale=0.4](y1){};
\node at (13+13.5,6+-5) [circle,draw=black,fill=black,scale=0.4](y2){};
\node at (13+13.5,6+-2) [circle,draw=black,fill=black,scale=0.4](y3){};
\node at (13+13.5,6+-1) [circle,draw=black,fill=black,scale=0.4](y4){};
\node at (13+13.5,6+0) [circle,draw=black,fill=black,scale=0.4](y5){};
\node at (8+13.5,6+-4) [circle,draw=black,fill=black,scale=0.4](z1){};
\node at (8+13.5,6+-2) [circle,draw=black,fill=black,scale=0.4](z2){};
\node at (8+13.5,6+-1) [circle,draw=black,fill=black,scale=0.4](z3){};
\node at (10.5+13.5,6+-4) (){$H(1,2,5,3,2)$};
\draw  (x1) -- (x2) -- (x3) -- (x4) -- (x5) -- (x6) -- (y1) -- (y2);
\draw  (x6) -- (y3) -- (y4) -- (y5);
\draw  (z1) -- (x1) -- (z2) -- (z3);
\draw  (12.5+13.5,6+-5.5) -- (y2) -- (13.5+13.5,6+-5.5);
\draw  (12.5+13.5,6+.5) -- (y5) -- (13.5+13.5,6+.5);
\draw  (7.5+13.5,6+-4.5) -- (z1) -- (8.5+13.5,6+-4.5);
\draw  (7.5+13.5,6+-.5) -- (z3) -- (8.5+13.5,6+-.5);

% \node at (5,-7) [circle,draw=black,scale=0.8](t1){};
% \node at (6,-7) [circle,draw=black,fill=black,scale=0.4](t2){};
% \node at (7,-7) [circle,draw=black,fill=black,scale=0.4](t3){};
% \node at (8,-7) [circle,draw=black,fill=black,scale=0.4](t4){};
% \node at (9,-7) [circle,draw=black,scale=0.8](t5){};
% \node at (4.4,-6.5) [circle,draw=black,fill=black,scale=0.4](h1){};
% \node at (3.6,-6.5) [circle,draw=black,fill=black,scale=0.4](h2){};
% \node at (3,-7) [circle,draw=black,fill=black,scale=0.4](h3){};
% \node at (4.4,-7.5) [circle,draw=black,fill=black,scale=0.4](h5){};
% \node at (3.6,-7.5) [circle,draw=black,fill=black,scale=0.4](h4){};
% \node at (9.6,-6.5) [circle,draw=black,fill=black,scale=0.4](h'1){};
% \node at (10.4,-6.5) [circle,draw=black,fill=black,scale=0.4](h'2){};
% \node at (11,-7) [circle,draw=black,fill=black,scale=0.4](h'3){};
% \node at (9.6,-7.5) [circle,draw=black,fill=black,scale=0.4](h'5){};
% \node at (10.4,-7.5) [circle,draw=black,fill=black,scale=0.4](h'4){};
% \node at (7,-8) (){$H^*(2,2,4,2,2)$};
% \draw  (t1) -- (h1) -- (h2) -- (h3)-- (h4) -- (h5) -- (t1);
% \draw  (t1) -- (t2) -- (t3) -- (t4) -- (t5);
% \draw  (t5) -- (h'1) -- (h'2) -- (h'3)-- (h'4) -- (h'5) -- (t5);
% \draw  (h2) -- (3.5,-6);
% \draw  (h3) -- (2.5,-7);
% \draw  (h4) -- (3.5,-8);
% \draw  (h'2) -- (10.5,-6);
% \draw  (h'3) -- (11.5,-7);
% \draw  (h'4) -- (10.5,-8);
\end{tikzpicture}

\end{center}
\caption{\label{handles}Examples of handles in a graph, from left to right: handle, parity handle, S-handle and H-handle (white vertices: end-vertices of the handle).}
\end{figure}

A handle $H$ is {\em $(a,b)$-reducible} in a graph $G$ if any $(a,b)$-coloring of $G-\int(H)$ can be extended to a $(a,b)$-coloring of $G$, possibly, for S- and H-handles, by modifying the color sets of some vertices of degree 2 of $H$ (other than those of $\int(H)$). %In other words, $H$ is $(a,b)$-reducible in $G$ if whatever the color sets of the end-vertices of the handle, there always exists a $(a,b)$-coloring of $\int(H)$.

In order to characterize graphs of list chromatic number 2, Erdös, Rubin and Taylor~\cite{ERT79} defined the core of a graph as the graph obtained after iteratively removing vertices of degree 1.
In the same vein, for a graph $G$ and a family $\mF$ of handles in $G$, we define  $\core_{\mF}(G)$ as any (induced) subgraph obtained after successively removing vertices of degree 0 and 1 and vertices of $\int(H)$ for each handle $H\in \mF$ until no more degree 0 or 1 vertex nor handle of $\mF$ remains. 
By the definition of reducibility, we immediately have the following result:
\begin{theorem}
For any graph $G$ and any family $\mF$ of $(a,b)$-reducible handles in $G$, $$G \ (a,b)\text{-colorable} \Leftrightarrow \core_{\mF}(G) \ (a,b)\text{-colorable.}$$
\end{theorem}

%\subsection{Handle reductions}
%The results of Section~\ref{secPath} immediately imply the following:
The proofs of the following results, necessitating long case analysis, are presented in Sections 4 and 6.
\begin{theorem}%[\cite{GT20}]
\label{ThmP}
For any graph $G$ and any integers $b,e$ such that $b\ge 2$ and $e<b$, any handle $P(n)$ with $n\ge \even(2b/e)$ is $(2b+e,b)$-reducible in $G$ and any parity handle $PP(n)$ with $n\ge 2$ is $(2b+e,b)$-reducible in $G$.
\end{theorem}

In order to present our reducibility results on S-handles and H-handles, we first need to define an ordering among them.
For two integer vectors $\vect{v}=(v_1,v_2,\ldots, v_k), \vect{v'}=(v'_1,v'_2,\ldots, v'_k)$ of $\mathbb{N}^k$ we consider the natural order: $\vect{v}\le \vect{v'}$ if and only if $ \forall i\in\{1,\ldots ,k\}$, $v_i\le v'_i$. Therefore, we will say that an S-handle $S(n_1,n_2,n_3)$ is {\em smaller} than (or equal to)  an S-handle $S(n'_1,n'_2,n'_3)$ if $(n'_1,n'_2,n'_3)\ge  (n_1,n_2,n_3)$  and similarly for H-handles. 

%The following reducibility results can be proved by a long case analysis. These proofs can be found in the long version~\cite{GT20} of the present paper. %~\ref{Pm} and \ref{Sm}.

\begin{theorem}%[\cite{GT20}]
\label{ThmS}
For any graph $G$ and any integers $b,e,k$ such that $b\ge \max\{2, e+1,k\}$, 
$S(\even(2b/e)-1,2,1)$ is a smallest $(2b+e,b)$-reducible S-handle and $S(2b-k,k,k)$ is a smallest $(2b+1,b)$-reducible S-handle in $G$.
\end{theorem}

\begin{theorem}%[\cite{GT20}]
\label{ThmH}
 For any graph $G$ and any integers $b,e$ such that $b\ge 2$ and $e<b$, 
 $H(1,2,\even(2b/e)-2,2,1)$ is a smallest $(2b+e,b)$-reducible H-handle and the following H-handles are  smallest $(2b+1,b)$-reducible H-handles in $G$:
 
 \begin{itemize}
 \item $H(2,2,2b-3,2,2)$, $H(1,2,2b-3,3,2)$, $H(1,4,2b-3,2,2)$;
 \item $H(1,2,2b-4,4,3)$, $H(1,4,2b-4,3,3)$. 
 \end{itemize}

\end{theorem}

Remark that the above Theorem may be completed with the handle $H(2,2,2b-4,3,2)$ that seems to be the smallest $(2b+1,b)$-reducible H-handle with $n_1=2$ and $n=2b-4$. Its reducibility was checked by computer for $b=4$ but a general proof for any $b$ seems to necessitate a very long case analysis and, as the case $b=4$ is sufficient for Conjecture~\ref{mcdconj}, we decided to skip this proof.
Remark also that we do not try to go below $2b-4$ for the third parameter of a H-handle since these results will be mainly used for the case $(a,b)=(9,4)$ in Section 3, and for these values, we have $2b-4=4$. Hence, going below 4 will result in a H-handle with $n\le n_3$ (i.e., the handle is not 'centered' on the longest path). Nevertheless, similar results can be obtained for lower values of $n$ (and greater values of $b$). For instance, we have found by computer these complementary reducible H-handles for $(a,b)=(9,4)$:

\begin{proposition}\label{PropH}
 For any graph $G$, the H-handles $H(2,2,4,3,2)$, $H(2,3,3,4,3)$ and $H(2,4,3,3,3)$ are all three $(9,4)$-reducible in $G$.
\end{proposition}
\begin{proof}
 By computer or straightforward (but long) case analysis.
\end{proof}

It can be shown (but the proof is tedious) that for the above handles, the core is in fact unique, i.e., whatever the order of the reductions made, the process will end with the same graph. 
\begin{remark}[Unicity of the core]
For any graph $G$ and the family $\mF$ of all handles from Theorems~\ref{ThmP},~\ref{ThmS},~\ref{ThmH} for fixed $b\ge 2$ and $e< b$ in $G$, $\core_{\mF}(G)$ is unique.
\end{remark}
%The idea is to prove that if at some step of reducing the graph in order to obtain a core we have the choice of reducing at least two handles, then the order in which we do these reductions will not change the obtained core at the end. For instance, suppose that $b=3, e=1$, and $G$ is a graph with two reducible handles $H$ and $H'$. Then either $H$ and $H'$ are disjoint and hence the order in which we reduce them has no importance, or they share a common end-point. Suppose that $H$ is a handle $P(6)$ with end-vertices $x$ and $y$ and that the end-vertices of $H'$ are $y$ and $z$ and that the other induced path starting from $y$ in $G$ is of length $\ell\ge 1$. Then if $H'$ is reduced before $H$, the length of $H$ will increase to $6+\ell$ and it will thus remain reducible by Theorem~\ref{ThmP}. The idea is similar if $H$ is an S-handle or H-handle. For example, let us now suppose that $H$ is handle $H(1,4,2b-4,3,3)$ and $b\ge 4, e=1$. Then, depending on the subpath of $H$ that is removed by removing the interior of $H'$, there will remain a handle $P(2b)$ or $S(2b-3,3,3)$ or $S(2b-1,4,1)$, each being reducible in $G$.

\section{Multicoloring triangle-free induced subgraphs of the triangular lattice}
A finite triangle-free induced subgraph of the triangular grid will be called a {\em hexagonal graph}.

In a similar way to the method used by Havet~\cite{hav01}, i.e., starting from a degree 3 vertex in a 'corner' of the graph and exploring the configurations around it to prove that a handle from $\mF$ is present, we can prove the following:
\begin{theorem}%[\cite{GT20}]\
label{handle}
Let $G$ be a hexagonal graph. For each of the following three families of handles we have $\core_{\mF}(G)=\emptyset$:

\begin{enumerate}
\item $\mF=\{P(2)\}$; 
\item $\mF=\{P(4), PP(3)\}$;
\item $\mF=\{P(6), PP(3), PP(4), PP(5), S(5,2,1) , H(1,2,4,2,1)\}$.
\end{enumerate}
\end{theorem}

This theorem along with Theorems~\ref{ThmP}, \ref{ThmS}, \ref{ThmH} allow to prove the known results~\cite{hav01} for $(a,b)$-colorability of hexagonal graphs for $b=1, 2$ and 3 in a unified way:
\begin{corollary}\label{corol}
Any hexagonal graph is $(3,1)$-colorable, $(5,2)$-colorable and $(7,3)$-colorable.
\end{corollary}

Note that $(7,3)$-colorability is enough as it implies $(5,2)$-colorability and $(3,1)$-colorability.

Now, we turn our attention to $(9,4)$-colorings.

Consider the set of handles $\mF_{9,4}=\mathcal{P}\cup \mathcal{S}\cup \mathcal{H}$, with\\
$\mathcal{P}= \{P(8), PP(3), PP(4), PP(5), PP(6), PP(7)\}$,\\ $\mathcal{S}= \{ S(7,2,1) , S(6,2,2), S(5,3,3), S(4,4,4)\}$, and\\  $\mathcal{H}=\{H(1,2,6,2,1), H(1,2,5,3,2),H(1,4,5,2,2), H(2,2,5,2,2), H(2,2,4,3,2)$,\\ $H(1,2,4,4,3), H(1,4,4,3,3), H(2,3,3,4,3), H(2,4,3,3,3)\}$.

\begin{figure}[ht]
\centering
\includegraphics[width=4cm]{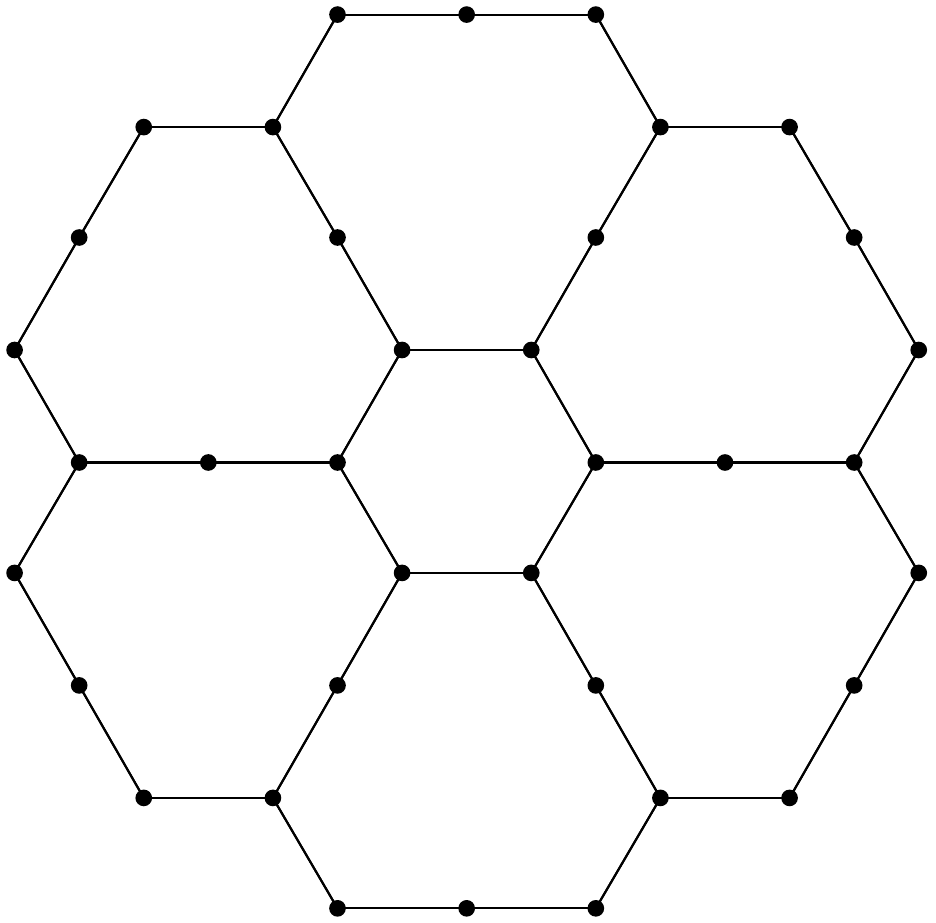}\hspace*{0.8cm}
\includegraphics[width=3.5cm]{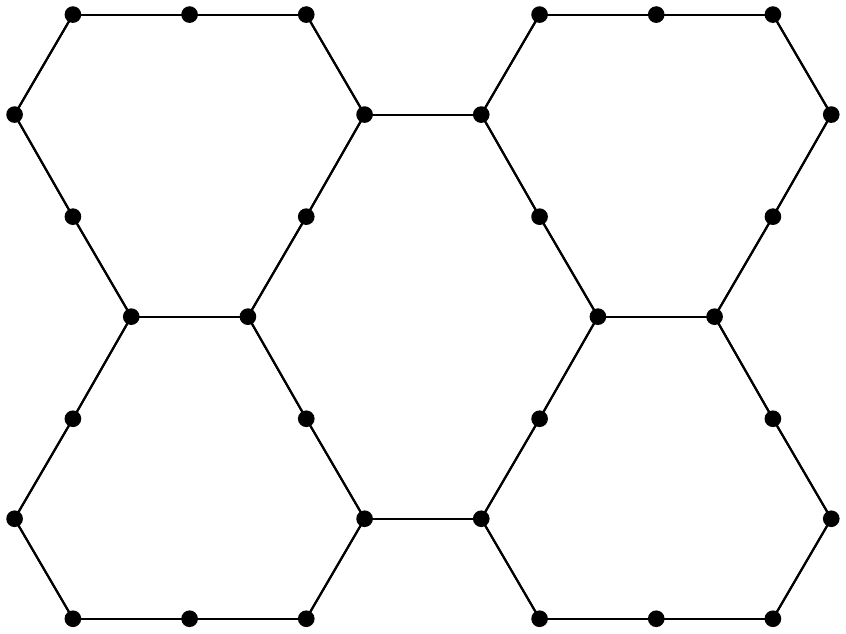}\hspace*{0.8cm}
\includegraphics[width=4cm]{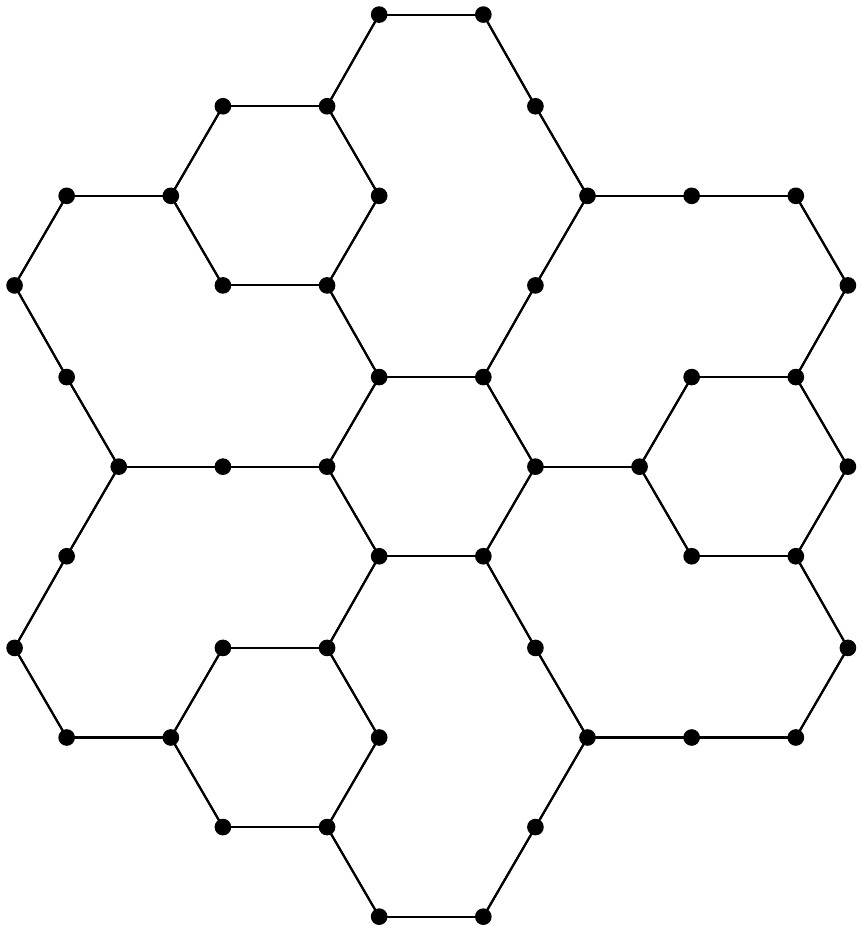}
\caption{\label{counter}Three hexagonal graphs that contain no handle from $\mathcal{P}\cup\mathcal{S}$.}
\end{figure}

Note that, by Theorems~\ref{ThmP}, \ref{ThmS}, \ref{ThmH} and Proposition~\ref{PropH}, the handles of $\mF_{9,4}$ are all $(9,4)$-reducible in any graph. Moreover, all the handles of $\mF_{9,4}$ are necessary to reduce hexagonal graphs since we have examples of hexagonal graphs for which the core is not empty if we remove one of the handles from $\mF_{9,4}$. Figure~\ref{counter} shows three examples of hexagonal graphs that only possess handles from $\mathcal{H}$ (the graph on the left contains only handles $H(2,4,4,4,2)$ (hence also $H(2,2,4,3,2)$), the one on the center only handles $H(1,2,6,2,1)$ and  the one on the right only handles $H(2,2,4,4,2)$.% and the one on the right only handles $H(2,2,4,4,2)$.

\subsection{More $(9,4)$-reducible handles}
In order to (try to) prove Conjecture~\ref{mcdconj}, the computational experiments we made lead us to find a set $\mF'_{9,4}$ of more sophisticated $(9,4)$-reducible configurations that we call {\em Cyclic H-handles}. The 25 configurations of $\mF'_{9,4}$ are presented in Figure~\ref{fig:newhandles}. These configurations are in fact H-handles with some additional constraints (other paths of prescribed length between some ports). Note that each of these configurations is reduced by removing the interior of the central path (between big white vertices). Note also that sometimes the distance between two ports along this central path may be lower than 8, meaning that if the central path is removed, there will be no more constraints on the 4-color sets of the two ports (like for the top left configuration of Figure~\ref{fig:newhandles} where the bottom left and bottom right ports are at distance 7). But in this case, it can be observed that there is always a second path between the two ports of the same parity and with lower than or equal length than the one going through the central path.

\begin{figure}[H]
%\begin{center}
\centering
\def\ax{0}\def\ay{24}
\def\bx{10}\def\by{24}
\def\cx{20}\def\cy{24}
\def\lx{30}\def\ly{24}
\def\dx{0}\def\dy{18}
\def\ex{10}\def\ey{18}
\def\fx{20}\def\fy{18}
\def\gx{30}\def\gy{18}
\def\hx{0}\def\hy{12}
\def\ix{10}\def\iy{12}
\def\jx{20}\def\jy{12}
\def\kx{30}\def\ky{12}
\def\wx{0}\def\wy{6}
\def\nx{10}\def\ny{6}
\def\ox{20}\def\oy{6}
\def\px{30}\def\py{6}
\def\qx{0}\def\qy{0}
\def\rx{10}\def\ry{0}
\def\sx{20}\def\sy{0}
\def\tx{30}\def\ty{0}
\def\ux{0}\def\uy{-6}
\def\vx{10}\def\vy{-6}
\def\xx{20}\def\xy{-6}
\def\yx{30}\def\yy{-6}
\def\zx{0}\def\zy{-12}
\begin{tikzpicture}[scale=0.4]
\node at (\ax+2,\ay) [circle,draw=black,scale=0.8](t1){};
\node at (\ax+5,\ay) [circle,draw=black,scale=0.8](t2){};
\node at (\ax+.6,\ay+1) [circle,draw=black,fill=black,scale=0.5](h1){};
\node at (\ax+2,\ay-1.5) [circle,draw=black,fill=black,scale=0.5](h2){};
\node at (\ax+6.4,\ay+1) [circle,draw=black,fill=black,scale=0.5](h3){};
\node at (\ax+5,\ay-1.5) [circle,draw=black,fill=black,scale=0.5](h4){};
\node at (\ax+6.4,\ay-2.5) [circle,draw=black,fill=black,scale=0.5](h5){};
\node at (\ax+.6,\ay-2.5) [circle,draw=black,fill=black,scale=0.5](h0){};
\draw  (t1) -- (h1) node[midway, above] {\small 2};
\draw  (h2)-- (t1) node[midway, left] {\small 1};
\draw [ultra thick] (t1) -- (t2) node[midway, above] {\small 2};
\draw  (t2) -- (h3) node[midway, above] {\small 4};
\draw  (t2) -- (h4) node[midway, right] {\small 1};
\draw [thick, style=dashed]  (h2) -- (h4)  node[midway, below] {\small 2};
\draw  (h2)-- (h0) node[midway, below] {\small 1};
\draw  (h5) -- (h4) node[midway, below] {\small 1};

\node at (\bx+2,\by) [circle,draw=black,scale=0.8](t1){};
\node at (\bx+5,\by) [circle,draw=black,scale=0.8](t2){};
\node at (\bx+.6,\by+1) [circle,draw=black,fill=black,scale=0.5](h1){};
\node at (\bx+2,\by-1.5) [circle,draw=black,fill=black,scale=0.5](h2){};
\node at (\bx+6.4,\by+1) [circle,draw=black,fill=black,scale=0.5](h3){};
\node at (\bx+5,\by-1.5) [circle,draw=black,fill=black,scale=0.5](h4){};
\node at (\bx+6.4,\by-2.5) [circle,draw=black,fill=black,scale=0.5](h5){};
\node at (\bx+.6,\by-2.5) [circle,draw=black,fill=black,scale=0.5](h0){};
\draw  (t1) -- (h1) node[midway, above] {\small 3};
\draw  (h2)-- (t1) node[midway, left] {\small 2};
\draw  (h2)-- (h0) node[midway, below] {\small 1};
%\draw [thick, style=dashed] (h1) .. controls+(-0.6,-1) .. (h2) node[midway, right] {2};
\draw [ultra thick] (t1) -- (t2) node[midway, above] {\small 2};
\draw  (t2) -- (h3) node[midway, above] {\small 3};
\draw  (h5) -- (h4) node[midway, below] {\small 1};
\draw  (t2) -- (h4) node[midway, right] {\small 2};
\draw [thick, style=dashed]  (h2) -- (h4)  node[midway, below] {\small 2,4 or 6};

\node at (\cx+2,\cy) [circle,draw=black,scale=0.8](t1){};
\node at (\cx+5,\cy) [circle,draw=black,scale=0.8](t2){};
\node at (\cx+.6,\cy+1) [circle,draw=black,fill=black,scale=0.5](h1){};
\node at (\cx+2,\cy-1.5) [circle,draw=black,fill=black,scale=0.5](h2){};
\node at (\cx+6.4,\cy+1) [circle,draw=black,fill=black,scale=0.5](h3){};
\node at (\cx+5,\cy-1.5) [circle,draw=black,fill=black,scale=0.5](h4){};
\node at (\cx+6.4,\cy-2.5) [circle,draw=black,fill=black,scale=0.5](h5){};
\node at (\cx+.6,\cy-2.5) [circle,draw=black,fill=black,scale=0.5](h0){};
\draw  (t1) -- (h1) node[midway, above] {\small 2};
\draw  (h2)-- (t1) node[midway, left] {\small 2};
\draw  (h2)-- (h0) node[midway, below] {\small 1};
%\draw [thick, style=dashed] (h1) .. controls+(-0.6,-1) .. (h2) node[midway, right] {2};
\draw [ultra thick] (t1) -- (t2) node[midway, above] {\small 2};
\draw  (t2) -- (h3) node[midway, above] {\small 4};
\draw  (h5) -- (h4) node[midway, below] {\small 1};
\draw  (t2) -- (h4) node[midway, right] {\small 2};
\draw [thick, style=dashed]  (h2) -- (h4)  node[midway, below] {\small 2};

\node at (\dx+2,\dy) [circle,draw=black,scale=0.8](t1){};
\node at (\dx+5,\dy) [circle,draw=black,scale=0.8](t2){};
\node at (\dx+.6,\dy+1) [circle,draw=black,fill=black,scale=0.5](h1){};
\node at (\dx+2,\dy-1.5) [circle,draw=black,fill=black,scale=0.5](h2){};
\node at (\dx+6.4,\dy+1) [circle,draw=black,fill=black,scale=0.5](h3){};
\node at (\dx+5,\dy-1.5) [circle,draw=black,fill=black,scale=0.5](h4){};
\node at (\dx+6.4,\dy-2.5) [circle,draw=black,fill=black,scale=0.5](h5){};
\node at (\dx+.6,\dy-2.5) [circle,draw=black,fill=black,scale=0.5](h0){};
\draw  (t1) -- (h1) node[midway, above] {\small 5};
\draw  (h2)-- (t1) node[midway, left] {\small 1};
\draw [ultra thick] (t1) -- (t2) node[midway, above] {\small 2};
\draw  (t2) -- (h3) node[midway, above] {\small 2};
\draw  (t2) -- (h4) node[midway, right] {\small 2};
\draw [thick, style=dashed]  (h2) -- (h4)  node[midway, below] {\small 1};
\draw  (h2)-- (h0) node[midway, below] {\small 2};
\draw  (h5) -- (h4) node[midway, below] {\small 2};

\node at (\fx+2,\fy) [circle,draw=black,scale=0.8](t1){};
\node at (\fx+5,\fy) [circle,draw=black,scale=0.8](t2){};
\node at (\fx+.6,\fy+1) [circle,draw=black,fill=black,scale=0.5](h1){};
\node at (\fx+2,\fy-1.5) [circle,draw=black,fill=black,scale=0.5](h2){};
\node at (\fx+6.4,\fy+1) [circle,draw=black,fill=black,scale=0.5](h3){};
\node at (\fx+5,\fy-1.5) [circle,draw=black,fill=black,scale=0.5](h4){};
\draw  (t1) -- (h1) node[midway, above] {\small 2};
\draw  (h2)-- (t1) node[midway, right] {\small 2};
\draw [ultra thick] (t1) -- (t2) node[midway, above] {\small 3};
\draw  (t2) -- (h3) node[midway, above] {\small 2};
\draw  (t2) -- (h4) node[midway, left] {\small 2};
\draw [thick, style=dashed]  (h2) -- (h4)  node[midway, below] {\small 3};
\draw [thick, style=dashed]  (h3) .. controls+(0.4,-1.7) .. (h4)  node[midway, left] {\small 2};
\draw [thick, style=dashed]  (h1) .. controls+(-.5,-1.6) .. (h2)  node[midway, right] {\small 2};

\node at (\gx+2,\gy) [circle,draw=black,scale=0.8](t1){};
\node at (\gx+5,\gy) [circle,draw=black,scale=0.8](t2){};
\node at (\gx+.6,\gy+1) [circle,draw=black,fill=black,scale=0.5](h1){};
\node at (\gx+2,\gy-1.5) [circle,draw=black,fill=black,scale=0.5](h2){};
\node at (\gx+6.4,\gy+1) [circle,draw=black,fill=black,scale=0.5](h3){};
\node at (\gx+5,\gy-1.5) [circle,draw=black,fill=black,scale=0.5](h4){};
\node at (\gx+6.4,\gy-2.5) [circle,draw=black,fill=black,scale=0.5](h5){};
\node at (\gx+.6,\gy-2.5) [circle,draw=black,fill=black,scale=0.5](h0){};
\draw  (t1) -- (h1) node[midway, above] {\small 2};
\draw  (h2)-- (t1) node[midway, left] {\small 1};
\draw [ultra thick] (t1) -- (t2) node[midway, above] {\small 3};
\draw  (t2) -- (h3) node[midway, above] {\small 5};
\draw  (t2) -- (h4) node[midway, right] {\small 1};
\draw [thick, style=dashed]  (h2) -- (h4)  node[midway, below] {\small 3};
\draw  (h2)-- (h0) node[midway, below] {\small 1};
\draw  (h5) -- (h4) node[midway, below] {\small 2};

\node at (\ex+2,\ey) [circle,draw=black,scale=0.8](t1){};
\node at (\ex+5,\ey) [circle,draw=black,scale=0.8](t2){};
\node at (\ex+.6,\ey+1) [circle,draw=black,fill=black,scale=0.5](h1){};
\node at (\ex+2,\ey-1.5) [circle,draw=black,fill=black,scale=0.5](h2){};
\node at (\ex+6.4,\ey+1) [circle,draw=black,fill=black,scale=0.5](h3){};
\node at (\ex+5,\ey-1.5) [circle,draw=black,fill=black,scale=0.5](h4){};
\node at (\ex+6.4,\ey-2.5) [circle,draw=black,fill=black,scale=0.5](h5){};
\node at (\ex+.6,\ey-2.5) [circle,draw=black,fill=black,scale=0.5](h0){};
\draw  (t1) -- (h1) node[midway, above] {\small 3};
\draw  (h2)-- (t1) node[midway, left] {\small 1};
\draw [ultra thick] (t1) -- (t2) node[midway, above] {\small 3};
\draw  (t2) -- (h3) node[midway, above] {\small 4};
\draw  (t2) -- (h4) node[midway, right] {\small 1};
\draw [thick, style=dashed]  (h2) -- (h4)  node[midway, below] {\small 3};
\draw  (h2)-- (h0) node[midway, below] {\small 2};
\draw  (h5) -- (h4) node[midway, below] {\small 2};

\node at (\hx+2,\hy) [circle,draw=black,scale=0.8](t1){};
\node at (\hx+5,\hy) [circle,draw=black,scale=0.8](t2){};
\node at (\hx+.6,\hy+1) [circle,draw=black,fill=black,scale=0.5](h1){};
\node at (\hx+2,\hy-1.5) [circle,draw=black,fill=black,scale=0.5](h2){};
\node at (\hx+6.4,\hy+1) [circle,draw=black,fill=black,scale=0.5](h3){};
\node at (\hx+5,\hy-1.5) [circle,draw=black,fill=black,scale=0.5](h4){};
\node at (\hx+6.4,\hy-2.5) [circle,draw=black,fill=black,scale=0.5](h5){};
\node at (\hx+.6,\hy-2.5) [circle,draw=black,fill=black,scale=0.5](h0){};
\draw  (t1) -- (h1) node[midway, above] {\small 3};
\draw  (h2)-- (t1) node[midway, left] {\small 1};
\draw [ultra thick] (t1) -- (t2) node[midway, above] {\small 3};
\draw  (t2) -- (h3) node[midway, above] {\small 4};
\draw  (t2) -- (h4) node[midway, right] {\small 2};
\draw [thick, style=dashed]  (h2) -- (h4)  node[midway, below] {\small 3};
\draw  (h2)-- (h0) node[midway, below] {\small 1};
\draw  (h5) -- (h4) node[midway, below] {\small 2};

\node at (\ix+2,\iy) [circle,draw=black,scale=0.8](t1){};
\node at (\ix+5,\iy) [circle,draw=black,scale=0.8](t2){};
\node at (\ix+.6,\iy+1) [circle,draw=black,fill=black,scale=0.5](h1){};
\node at (\ix+2,\iy-1.5) [circle,draw=black,fill=black,scale=0.5](h2){};
\node at (\ix+6.4,\iy+1) [circle,draw=black,fill=black,scale=0.5](h3){};
\node at (\ix+5,\iy-1.5) [circle,draw=black,fill=black,scale=0.5](h4){};
\node at (\ix+6.4,\iy-2.5) [circle,draw=black,fill=black,scale=0.5](h5){};
\node at (\ix+.6,\iy-2.5) [circle,draw=black,fill=black,scale=0.5](h0){};
\draw  (t1) -- (h1) node[midway, above] {\small 2};
\draw  (h2)-- (t1) node[midway, left] {\small 1};
\draw [ultra thick] (t1) -- (t2) node[midway, above] {\small 3};
\draw  (t2) -- (h3) node[midway, above] {\small 4};
\draw  (t2) -- (h4) node[midway, right] {\small 3};
\draw [thick, style=dashed]  (h2) -- (h4)  node[midway, below] {\small 2};
\draw  (h2)-- (h0) node[midway, below] {\small 1};
\draw  (h5) -- (h4) node[midway, below] {\small 1};

\node at (\jx+2,\jy) [circle,draw=black,scale=0.8](t1){};
\node at (\jx+5,\jy) [circle,draw=black,scale=0.8](t2){};
\node at (\jx+.6,\jy+1) [circle,draw=black,fill=black,scale=0.5](h1){};
\node at (\jx+2,\jy-1.5) [circle,draw=black,fill=black,scale=0.5](h2){};
\node at (\jx+6.4,\jy+1) [circle,draw=black,fill=black,scale=0.5](h3){};
\node at (\jx+5,\jy-1.5) [circle,draw=black,fill=black,scale=0.5](h4){};
\node at (\jx+6.4,\jy-2.5) [circle,draw=black,fill=black,scale=0.5](h5){};
\node at (\jx+.6,\jy-2.5) [circle,draw=black,fill=black,scale=0.5](h0){};
\draw  (t1) -- (h1) node[midway, above] {\small 2};
\draw  (h2)-- (t1) node[midway, left] {\small 1};
\draw [ultra thick] (t1) -- (t2) node[midway, above] {\small 4};
\draw  (t2) -- (h3) node[midway, above] {\small 2};
\draw  (t2) -- (h4) node[midway, right] {\small 1};
\draw [thick, style=dashed]  (h2) -- (h4)  node[midway, below] {\small 4};
\draw  (h2)-- (h0) node[midway, below] {\small 1};
\draw  (h5) -- (h4) node[midway, below] {\small 1};

\node at (\kx+2,\ky) [circle,draw=black,scale=0.8](t1){};
\node at (\kx+5,\ky) [circle,draw=black,scale=0.8](t2){};
\node at (\kx+.6,\ky+1) [circle,draw=black,fill=black,scale=0.5](h1){};
\node at (\kx+2,\ky-1.5) [circle,draw=black,fill=black,scale=0.5](h2){};
\node at (\kx+6.4,\ky+1) [circle,draw=black,fill=black,scale=0.5](h3){};
\node at (\kx+5,\ky-1.5) [circle,draw=black,fill=black,scale=0.5](h4){};
\node at (\kx+.6,\ky-2.5) [circle,draw=black,fill=black,scale=0.5](h0){};
\draw  (t1) -- (h1) node[midway, above] {\small 1};
\draw  (h2)-- (t1) node[midway, right] {\small 2};
\draw [ultra thick] (t1) -- (t2) node[midway, above] {\small 4};
\draw  (t2) -- (h3) node[midway, above] {\small 2};
\draw  (t2) -- (h4) node[midway, left] {\small 2};
\draw [thick, style=dashed]  (h2) -- (h4)  node[midway, below] {\small 2};
\draw  (h2)-- (h0) node[midway, below] {\small 1};
\draw [thick, style=dashed]  (h3) .. controls+(0.4,-1.7) .. (h4)  node[midway, left] {\small 2};

\node at (\lx+2,\ly) [circle,draw=black,scale=0.8](t1){};
\node at (\lx+5,\ly) [circle,draw=black,scale=0.8](t2){};
\node at (\lx+.6,\ly+1) [circle,draw=black,fill=black,scale=0.5](h1){};
\node at (\lx+2,\ly-1.5) [circle,draw=black,fill=black,scale=0.5](h2){};
\node at (\lx+6.4,\ly+1) [circle,draw=black,fill=black,scale=0.5](h3){};
\node at (\lx+5,\ly-1.5) [circle,draw=black,fill=black,scale=0.5](h4){};
\node at (\lx+6.4,\ly-2.5) [circle,draw=black,fill=black,scale=0.5](h5){};
\node at (\lx+.6,\ly-2.5) [circle,draw=black,fill=black,scale=0.5](h0){};
\draw  (t1) -- (h1) node[midway, above] {\small 2};
\draw  (h2)-- (t1) node[midway, left] {\small 2};
\draw  (h2)-- (h0) node[midway, below] {\small 1};
%\draw [thick, style=dashed] (h1) .. controls+(-0.6,-1) .. (h2) node[midway, right] {2};
\draw [ultra thick] (t1) -- (t2) node[midway, above] {\small 2};
\draw  (t2) -- (h3) node[midway, above] {\small 4};
\draw  (h5) -- (h4) node[midway, below] {\small 2};
\draw  (t2) -- (h4) node[midway, right] {\small 2};
\draw [thick, style=dashed]  (h2) -- (h4)  node[midway, below] {\small 4 or 6};

\node at (\wx+2,\wy) [circle,draw=black,scale=0.8](t1){};
\node at (\wx+5,\wy) [circle,draw=black,scale=0.8](t2){};
\node at (\wx+.6,\wy+1) [circle,draw=black,fill=black,scale=0.5](h1){};
\node at (\wx+2,\wy-1.5) [circle,draw=black,fill=black,scale=0.5](h2){};
\node at (\wx+6.4,\wy+1) [circle,draw=black,fill=black,scale=0.5](h3){};
\node at (\wx+5,\wy-1.5) [circle,draw=black,fill=black,scale=0.5](h4){};
\node at (\wx+6.4,\wy-2.5) [circle,draw=black,fill=black,scale=0.5](h5){};
\node at (\wx+.6,\wy-2.5) [circle,draw=black,fill=black,scale=0.5](h0){};
\draw  (t1) -- (h1) node[midway, above] {\small 1};
\draw  (h2)-- (t1) node[midway, left] {\small 2};
\draw [ultra thick] (t1) -- (t2) node[midway, above] {\small 4};
\draw  (t2) -- (h3) node[midway, above] {\small 3};
\draw  (t2) -- (h4) node[midway, right] {\small 2};
\draw [thick, style=dashed]  (h2) -- (h4)  node[midway, below] {\small 2,4 or 5};
\draw  (h2)-- (h0) node[midway, below] {\small 1};
\draw  (h5) -- (h4) node[midway, below] {\small 1};

\node at (\nx+2,\ny) [circle,draw=black,scale=0.8](t1){};
\node at (\nx+5,\ny) [circle,draw=black,scale=0.8](t2){};
\node at (\nx+.6,\ny+1) [circle,draw=black,fill=black,scale=0.5](h1){};
\node at (\nx+2,\ny-1.5) [circle,draw=black,fill=black,scale=0.5](h2){};
\node at (\nx+6.4,\ny+1) [circle,draw=black,fill=black,scale=0.5](h3){};
\node at (\nx+5,\ny-1.5) [circle,draw=black,fill=black,scale=0.5](h4){};
\node at (\nx+6.4,\ny-2.5) [circle,draw=black,fill=black,scale=0.5](h5){};
\node at (\nx+.6,\ny-2.5) [circle,draw=black,fill=black,scale=0.5](h0){};
\draw  (t1) -- (h1) node[midway, above] {\small 1};
\draw  (h2)-- (t1) node[midway, left] {\small 2};
\draw [ultra thick] (t1) -- (t2) node[midway, above] {\small 4};
\draw  (t2) -- (h3) node[midway, above] {\small 4};
\draw  (t2) -- (h4) node[midway, right] {\small 2};
\draw [thick, style=dashed]  (h2) -- (h4)  node[midway, below] {\small 3 or 4};
\draw  (h2)-- (h0) node[midway, below] {\small 1};
\draw  (h5) -- (h4) node[midway, below] {\small 1};

\node at (\ox+2,\oy) [circle,draw=black,scale=0.8](t1){};
\node at (\ox+5,\oy) [circle,draw=black,scale=0.8](t2){};
\node at (\ox+.6,\oy+1) [circle,draw=black,fill=black,scale=0.5](h1){};
\node at (\ox+2,\oy-1.5) [circle,draw=black,fill=black,scale=0.5](h2){};
\node at (\ox+6.4,\oy+1) [circle,draw=black,fill=black,scale=0.5](h3){};
\node at (\ox+5,\oy-1.5) [circle,draw=black,fill=black,scale=0.5](h4){};
\draw  (t1) -- (h1) node[midway, above] {\small 1};
\draw  (h2)-- (t1) node[midway, right] {\small 3};
\draw [ultra thick] (t1) -- (t2) node[midway, above] {\small 4};
\draw  (t2) -- (h3) node[midway, above] {\small 3};
\draw  (t2) -- (h4) node[midway, left] {\small 3};
\draw [thick, style=dashed]  (h3) .. controls+(0.4,-1.7) .. (h4)  node[midway, right] {\small 3 or 5};

\node at (\px+2,\py) [circle,draw=black,scale=0.8](t1){};
\node at (\px+5,\py) [circle,draw=black,scale=0.8](t2){};
\node at (\px+.6,\py+1) [circle,draw=black,fill=black,scale=0.5](h1){};
\node at (\px+3.5,\py-1) [circle,draw=black,fill=black,scale=0.5](h2){};
\node at (\px+6.4,\py+1) [circle,draw=black,fill=black,scale=0.5](h3){};
%\node at (\px+6.4,\py-1) [circle,draw=black,fill=black,scale=0.5](h4){};
\node at (\px+3.5,\py-2.4) [circle,draw=black,fill=black,scale=0.5](h5){};
\draw  (t1) -- (h1) node[midway, above] {\small 1};
\draw  (h2)-- (t1) node[midway, below] {\small 3};
%\draw [thick, style=dashed] (h1) .. controls+(-0.6,-1) .. (h2) node[midway, right] {2};
\draw [ultra thick] (t1) -- (t2) node[midway, above] {\small 4};
\draw  (t2) -- (h3) node[midway, above] {\small 3};
\draw  (t2) -- (h2) node[midway, below] {\small 2};
\draw  (h5) -- (h2) node[midway, right] {\small 1};

\node at (\qx+2,\qy) [circle,draw=black,scale=0.8](t1){};
\node at (\qx+5,\qy) [circle,draw=black,scale=0.8](t2){};
\node at (\qx+.6,\qy+1) [circle,draw=black,fill=black,scale=0.5](h1){};
\node at (\qx+2,\qy-1.5) [circle,draw=black,fill=black,scale=0.5](h2){};
\node at (\qx+6.4,\qy+1) [circle,draw=black,fill=black,scale=0.5](h3){};
\node at (\qx+5,\qy-1.5) [circle,draw=black,fill=black,scale=0.5](h4){};
\node at (\qx+6.4,\qy-2.5) [circle,draw=black,fill=black,scale=0.5](h5){};
\node at (\qx+.6,\qy-2.5) [circle,draw=black,fill=black,scale=0.5](h0){};
\draw  (t1) -- (h1) node[midway, above] {\small 2};
\draw  (h2)-- (t1) node[midway, left] {\small 2};
\draw [ultra thick] (t1) -- (t2) node[midway, above] {\small 4};
\draw  (t2) -- (h3) node[midway, above] {\small 2};
\draw  (t2) -- (h4) node[midway, right] {\small 2};
\draw [thick, style=dashed]  (h2) -- (h4)  node[midway, below] {\small 2};
\draw  (h2)-- (h0) node[midway, below] {\small 1};
\draw  (h5) -- (h4) node[midway, below] {\small 1};

\node at (\rx+2,\ry) [circle,draw=black,scale=0.8](t1){};
\node at (\rx+5,\ry) [circle,draw=black,scale=0.8](t2){};
\node at (\rx+.6,\ry+1) [circle,draw=black,fill=black,scale=0.5](h1){};
\node at (\rx+2,\ry-1.5) [circle,draw=black,fill=black,scale=0.5](h2){};
\node at (\rx+6.4,\ry+1) [circle,draw=black,fill=black,scale=0.5](h3){};
\node at (\rx+5,\ry-1.5) [circle,draw=black,fill=black,scale=0.5](h4){};
\draw  (t1) -- (h1) node[midway, above] {\small 2};
\draw  (h2)-- (t1) node[midway, right] {\small 2};
\draw [ultra thick] (t1) -- (t2) node[midway, above] {\small 4};
\draw  (t2) -- (h3) node[midway, above] {\small 2};
\draw  (t2) -- (h4) node[midway, left] {\small 2};
%\draw [thick, style=dashed]  (h2) -- (h4)  node[midway, below] {\small 3};
\draw [thick, style=dashed]  (h3) .. controls+(0.4,-1.7) .. (h4)  node[midway, left] {\small 2};
\draw [thick, style=dashed]  (h1) .. controls+(-.5,-1.6) .. (h2)  node[midway, right] {\small 2};

\node at (\sx+2,\sy) [circle,draw=black,scale=0.8](t1){};
\node at (\sx+5,\sy) [circle,draw=black,scale=0.8](t2){};
\node at (\sx+.6,\sy+1) [circle,draw=black,fill=black,scale=0.5](h1){};
\node at (\sx+2,\sy-1.5) [circle,draw=black,fill=black,scale=0.5](h2){};
\node at (\sx+6.4,\sy+1) [circle,draw=black,fill=black,scale=0.5](h3){};
\node at (\sx+5,\sy-1.5) [circle,draw=black,fill=black,scale=0.5](h4){};
\node at (\sx+6.4,\sy-2.5) [circle,draw=black,fill=black,scale=0.5](h5){};
\node at (\sx+.6,\sy-2.5) [circle,draw=black,fill=black,scale=0.5](h0){};
\draw  (t1) -- (h1) node[midway, above] {\small 1};
\draw  (h2)-- (t1) node[midway, left] {\small 2};
\draw [ultra thick] (t1) -- (t2) node[midway, above] {\small 5};
\draw  (t2) -- (h3) node[midway, above] {\small 2};
\draw  (t2) -- (h4) node[midway, right] {\small 1};
\draw [thick, style=dashed]  (h2) -- (h4)  node[midway, below] {\small 7};
\draw  (h2)-- (h0) node[midway, below] {\small 1};
\draw  (h5) -- (h4) node[midway, below] {\small 1};

\node at (\tx+2,\ty) [circle,draw=black,scale=0.8](t1){};
\node at (\tx+5,\ty) [circle,draw=black,scale=0.8](t2){};
\node at (\tx+.6,\ty+1) [circle,draw=black,fill=black,scale=0.5](h1){};
\node at (\tx+2,\ty-1.5) [circle,draw=black,fill=black,scale=0.5](h2){};
\node at (\tx+6.4,\ty+1) [circle,draw=black,fill=black,scale=0.5](h3){};
\node at (\tx+5,\ty-1.5) [circle,draw=black,fill=black,scale=0.5](h4){};
\node at (\tx+.6,\ty-2.5) [circle,draw=black,fill=black,scale=0.5](h0){};
\draw  (t1) -- (h1) node[midway, above] {\small 1};
\draw  (h2)-- (t1) node[midway, right] {\small 2};
\draw [ultra thick] (t1) -- (t2) node[midway, above] {\small 5};
\draw  (t2) -- (h3) node[midway, above] {\small 2};
\draw  (t2) -- (h4) node[midway, left] {\small 2};
\draw [thick, style=dashed]  (h2) -- (h4)  node[midway, below] {\small 3};
\draw  (h2)-- (h0) node[midway, below] {\small 1};
\draw [thick, style=dashed]  (h3) .. controls+(0.4,-1.7) .. (h4)  node[midway, left] {\small 2};

\node at (\ux+2,\uy) [circle,draw=black,scale=0.8](t1){};
\node at (\ux+5,\uy) [circle,draw=black,scale=0.8](t2){};
\node at (\ux+.6,\uy+1) [circle,draw=black,fill=black,scale=0.5](h1){};
\node at (\ux+2,\uy-1.5) [circle,draw=black,fill=black,scale=0.5](h2){};
\node at (\ux+6.4,\uy+1) [circle,draw=black,fill=black,scale=0.5](h3){};
\node at (\ux+5,\uy-1.5) [circle,draw=black,fill=black,scale=0.5](h4){};
\node at (\ux+6.4,\uy-2.5) [circle,draw=black,fill=black,scale=0.5](h5){};
\node at (\ux+.6,\uy-2.5) [circle,draw=black,fill=black,scale=0.5](h0){};
\draw  (t1) -- (h1) node[midway, above] {\small 1};
\draw  (h2)-- (t1) node[midway, left] {\small 2};
\draw [ultra thick] (t1) -- (t2) node[midway, above] {\small 5};
\draw  (t2) -- (h3) node[midway, above] {\small 3};
\draw  (t2) -- (h4) node[midway, right] {\small 1};
\draw [thick, style=dashed]  (h2) -- (h4)  node[midway, below] {\small 5};
\draw  (h2)-- (h0) node[midway, below] {\small 1};
\draw  (h5) -- (h4) node[midway, below] {\small 1};

\node at (\vx+2,\vy) [circle,draw=black,scale=0.8](t1){};
\node at (\vx+5,\vy) [circle,draw=black,scale=0.8](t2){};
\node at (\vx+.6,\vy+1) [circle,draw=black,fill=black,scale=0.5](h1){};
\node at (\vx+2,\vy-1.5) [circle,draw=black,fill=black,scale=0.5](h2){};
\node at (\vx+6.4,\vy+1) [circle,draw=black,fill=black,scale=0.5](h3){};
\node at (\vx+5,\vy-1.5) [circle,draw=black,fill=black,scale=0.5](h4){};
\node at (\vx+6.4,\vy-2.5) [circle,draw=black,fill=black,scale=0.5](h5){};
\node at (\vx+.6,\vy-2.5) [circle,draw=black,fill=black,scale=0.5](h0){};
\draw  (t1) -- (h1) node[midway, above] {\small 1};
\draw  (h2)-- (t1) node[midway, left] {\small 2};
\draw [ultra thick] (t1) -- (t2) node[midway, above] {\small 5};
\draw  (t2) -- (h3) node[midway, above] {\small 4};
\draw  (t2) -- (h4) node[midway, right] {\small 1};
\draw [thick, style=dashed]  (h2) -- (h4)  node[midway, below] {\small 2};
\draw  (h2)-- (h0) node[midway, below] {\small 1};
\draw  (h5) -- (h4) node[midway, below] {\small 1};

\node at (\xx+2,\xy) [circle,draw=black,scale=0.8](t1){};
\node at (\xx+5,\xy) [circle,draw=black,scale=0.8](t2){};
\node at (\xx+.6,\xy+1) [circle,draw=black,fill=black,scale=0.5](h1){};
\node at (\xx+3.5,\xy-1) [circle,draw=black,fill=black,scale=0.5](h2){};
\node at (\xx+6.4,\xy+1) [circle,draw=black,fill=black,scale=0.5](h3){};
%\node at (\px+6.4,\py-1) [circle,draw=black,fill=black,scale=0.5](h4){};
\node at (\xx+3.5,\xy-2.4) [circle,draw=black,fill=black,scale=0.5](h5){};
\draw  (t1) -- (h1) node[midway, above] {\small 1};
\draw  (h2)-- (t1) node[midway, below] {\small 1};
%\draw [thick, style=dashed] (h1) .. controls+(-0.6,-1) .. (h2) node[midway, right] {2};
\draw [ultra thick] (t1) -- (t2) node[midway, above] {\small 6};
\draw  (t2) -- (h3) node[midway, above] {\small 1};
\draw  (t2) -- (h2) node[midway, below] {\small 2};
\draw  (h5) -- (h2) node[midway, below right] {\small 1};

\node at (\yx+2,\yy) [circle,draw=black,scale=0.8](t1){};
\node at (\yx+5,\yy) [circle,draw=black,scale=0.8](t2){};
\node at (\yx+.6,\yy+1) [circle,draw=black,fill=black,scale=0.5](h1){};
\node at (\yx+2,\yy-1.5) [circle,draw=black,fill=black,scale=0.5](h2){};
\node at (\yx+6.4,\yy+1) [circle,draw=black,fill=black,scale=0.5](h3){};
\node at (\yx+5,\yy-1.5) [circle,draw=black,fill=black,scale=0.5](h4){};
\node at (\yx+6.4,\yy-2.5) [circle,draw=black,fill=black,scale=0.5](h5){};
\node at (\yx+.6,\yy-2.5) [circle,draw=black,fill=black,scale=0.5](h0){};
\draw  (t1) -- (h1) node[midway, above] {\small 1};
\draw  (h2)-- (t1) node[midway, left] {\small 1};
\draw [ultra thick] (t1) -- (t2) node[midway, above] {\small 6};
\draw  (t2) -- (h3) node[midway, above] {\small 2};
\draw  (t2) -- (h4) node[midway, right] {\small 1};
\draw [thick, style=dashed]  (h2) -- (h4)  node[midway, below] {\small 4 or 6};
\draw  (h2)-- (h0) node[midway, below] {\small 1};
\draw  (h5) -- (h4) node[midway, below] {\small 1};

\node at (\zx+2,\zy) [circle,draw=black,scale=0.8](t1){};
\node at (\zx+5,\zy) [circle,draw=black,scale=0.8](t2){};
\node at (\zx+.6,\zy+1) [circle,draw=black,fill=black,scale=0.5](h1){};
\node at (\zx+2,\zy-1.5) [circle,draw=black,fill=black,scale=0.5](h2){};
\node at (\zx+6.4,\zy+1) [circle,draw=black,fill=black,scale=0.5](h3){};
\node at (\zx+5,\zy-1.5) [circle,draw=black,fill=black,scale=0.5](h4){};
\node at (\zx+6.4,\zy-2.5) [circle,draw=black,fill=black,scale=0.5](h5){};
\node at (\zx+.6,\zy-2.5) [circle,draw=black,fill=black,scale=0.5](h0){};
\draw  (t1) -- (h1) node[midway, above] {\small 1};
\draw  (h2)-- (t1) node[midway, left] {\small 1};
\draw [ultra thick] (t1) -- (t2) node[midway, above] {\small 7};
\draw  (t2) -- (h3) node[midway, above] {\small 1};
\draw  (t2) -- (h4) node[midway, right] {\small 1};
\draw [thick, style=dashed]  (h2) -- (h4)  node[midway, below] {\small 4 or 6};
\draw  (h2)-- (h0) node[midway, below] {\small 1};
\draw  (h5) -- (h4) node[midway, below] {\small 1};
\end{tikzpicture}
%\end{center}
\caption{\label{fig:newhandles} $\mF'_{9,4}$: Cyclic $(9,4)$-reducible H-handles (numbers: path lengths; plain lines: induced paths with degree 2 interior vertices; dashed lines: induced paths with possibly vertices of degree greater than two).}
\end{figure}

\begin{proposition}
 For any graph $G$, every handle from $\mF'_{9,4}$ from Figure~\ref{fig:newhandles} is $(9,4)$-reducible in $G$.
\end{proposition}

\begin{proof}
By computer (see Section~\ref{sec3.2} for the method and link to the code).
\end{proof}
As we will see in Section~\ref{sec3.3}, we experimentally observed that the reducible configurations from $\mF_{9,4}\cup\mF'_{9,4}$ are sufficient to obtain an empty core starting from any hexagonal graph, i.e., to completely $(9,4)$-color it, except if the graph contains the graph of Figure~\ref{7flower} as a subgraph.

However, we further present in Table~\ref{tbSp} of Appendix~\ref{app} a new set $\mF''_{9,4}$ of (more complex) $(9,4)$-reducible configurations obtained by extending S-handles and H-handles on one or two ports. The $(9,4)$-reducibility of each of these configurations has been tested with the computer.

\subsection{Testing configurations with the computer}
\label{sec3.2}
In order to test that the configurations of $\mF'_{9,4}$ and $\mF''_{9,4}$ are each $(9,4)$-reducible, we used the computer.
The general process was to test that by assigning any 4-color-set on each of the ports of the configuration, it is always possible to complete the $(9,4)$-coloring of the other vertices of the configuration.
To reduce the complexity (and running time), we take into account the symmetry in colors. Let us illustrate on an example of a H-handle  (with 4 ports and possibly with cycles, like the ones of Figure~\ref{fig:newhandles}). There are ${9\choose{4}} =126$ 4-color-sets and 4 ports, hence $126_{5} = 252047376$ assignments to test. However, by permuting the colors, we can always consider that the first port is given the color set $C_1=\{1,2,3,4\}$, which we write as $C_1=1234$ for short. For the second port, we can restrict to 5 4-color-sets given by the vectors $\mathbf{v}=(i,4-i), i\in [0-4]$ where $i$ is the number of colors shared with $C_1$ (and $4-i$ is the number of colors in $[5-9]$). Hence we can assign to Port 2 the color-sets from $C_2=\{1234,1235,1256,1567,5678\}$ for instance.

Then for the third port, depending on the number of shared colors between the color-sets of Port 1 and Port 2, there can be 5, 9, 19, 14, or 22 4-color-sets:
%In details, writting the 4-color-sets as words of length 4 over the alphabet $[1-9]$, we can assign to ports 3 the 4-color-sets of set $C_3$ defined by:\\
%$C_1=\{1234\}$, \\ %=\{(0,4,0,0),(0,3,0,1),(0,2,0,2),(0,1,0,3),(0,0,0,4)\}$,\\

$C_3 = \left\{\begin{array}{ll}
               C_2, & if c(v_2)=1234,\\
               \{1234,1235,1236,1245,1246,1256,1267,1456,1467,1567,1678,&\\
               4567,4678,5678,6789\} & if c(v_2)=1235,\\
               \{1234,1235,1237,1256,1257,1278,1345,1347,1356,1357,1378,&\\
               1567,1578,1789,3456,3457,3478,3567,3578,3789,5678,5789\} & if c(v_2)=1256,\\
               \{1234,1235,1238,1256,1258,1289,1567,1568,1589,2345,2348,&\\
               2356,2358,2389,2567,2568,2589,5678,5689\} & if c(v_2)=1567,\\
               \{1234,1235,1239,1256,1259,1567,1569,5678,5679\} & if c(v_2)=5678.\\
              \end{array}\right.$
              
For the fourth port, we then test the 126 color-sets (this could also be optimized by looking at the sets chosen in $C_2$ and $C_3$ but we do not do it).

Thus, in total, we have reduced the total number of assignments to test from $126_{5}$ to $(5+15+22+19+9)126=8820$.

The Python program \texttt{TestReducibility.py} developed for checking reducibility of a configuration of type $S, H, S', S'', H', H''$ and those of Figure~\ref{fig:newhandles} can be accessed at \url{https://github.com/otogni/MCD-coloring.git}.

\subsection{Computational reducibility experiments}
\label{sec3.3}
We have performed computational experiments for testing the reduction tools defined in the previous section on hexagonal graphs for finding a $(9,4)$-coloring. The algorithms were coded in C++ and ran on an Intel Xeon CPU at $2.67$ GHz and 24 GB of memory. The C++ source code can be accessed at \url{https://github.com/otogni/MCD-coloring.git}.
%The computations are grouped into two sets: in the first one, $\core_{\mF_{9,4}}$ is computed for randomly generated hexagonal graphs while in the second one we compute $\core_{\mF_{9,4}\cup \mF'_{9,4}}$ on 'hard' instances of hexagonal graphs containing many induced 9-cycles.

\subsubsection*{Experiment 1: Random generation}

{\noindent\textbf{Graph generation: }} The graphs are generated randomly on a grid of size $\ell\times h$ (the triangular lattice is considered as a square lattice with diagonals) by choosing randomly the coordinates of a point and testing if the corresponding vertex can be added to the graph without creating a triangle (repeated $5\ell h$ times). In order to obtain the 'harder' instances of random hexagonal graphs, we then do a final pass in which we consider the points of the grid in sequence and test if they can be added to the graph. Hence the graphs obtained are maximal triangle-free subgraphs of the triangular lattice, i.e., no point in the area can be added without creating a triangle.

{\noindent\textbf{Reduction algorithm: }} it consists in testing, for each vertex $x$ in sequence, if $x$ is the end-vertex of a handle from $\mF_{9,4}$ or $\mF'_{9,4}$.

Table~\ref{tb1} presents some measures of Experiment 1: the mean number of nodes (degree-3 vertices) and time needed to generate and completely reduce the graph depending on the side-length of the grid $n$ (the square root of the grid' size). One can observe that even for a hexagonal graph with hundreds of thousands degree-3 vertices (and around 2M vertices in total), our reduction algorithm can still construct a $(9,4)$-coloring in a reasonable time, even if it is far from being optimized.

\begin{table}[h]
\begin{tabular}{|l|llllllll|}\hline
$n$ (grid size $n\times n$) & 10  & 25   & 50     & 100   & 200  & 500   & 1000  & 2000 \\
\# nodes      & 8   & 92   & 492    & 1763  & 7450 & 46722 & 188452 & 754485\\
time (s)      & 0.0022 & 0.0026 & 0.0046 & 0.026 & 0.31 & 12.3  & 185 & 3618\\  \hline   
\end{tabular}

\caption{\label{tb1}Mean number of nodes (degree 3 vertices) and mean total time needed for generation and reduction depending on the grid side size $n$.}
\end{table}

Since the graphs generated in Experiment 1 have all been reduced, we tried to build more complex and 'hard' to reduce hexagonal graphs. For this, our intuition tells us that hexagonal graphs with many 9-cycles will be harder to reduce. Experiments in this direction lead us to generate semi-random hexagonal graphs with many 9-cycles by using the flower graph depicted on the left of Figure~\ref{counter}, as explained in the following experiment.

\subsubsection*{Experiment 2: Semi-random generation}

{\noindent\textbf{Graph generation: }} The graphs $G_{p,d}$ are generated by starting from an empty grid of size $\ell\times h$ and then first putting flower graphs isomorphic to the graph depicted on the left of Figure~\ref{counter} in a quasi-regular pavement, i.e., a flower is positioned in coordinates $(pi+r_{ij}^1+5,pj+r_{ij}^2+5)$, for $0\le i<h/15$ and $0\le j< \ell/15$, with $p\ge 10$ being a parameter for controlling the distance between flowers (observing that $p=11$ produces a maximum number of hard instances) and $r_{ij}^1, r_{ij}^2$ being random integers between 0 and $d$. Note that $d$ has to be small enough compared to $p$ in order to avoid triangles. The obtained graph is then completed randomly in order to obtain a maximal induced triangle-free subgraph of the triangular grid, as described in Experiment 1.  

{\noindent\textbf{Reduction algorithm:}} it consists in testing, for each vertex $x$ in sequence, if $x$ is the end-vertex of a handle from $\mF_{9,4}$ or  $\mF'_{9,4}$.

Again, the Reduction algorithm has been tested on hundreds of millions of semi-random graphs (mainly for grid size $50\times 50$) and it allows to reduce completely all the graphs except for $G_{9,0}$ which is always a subgraph of the regular lattice containing faces of length 6 and 9 obtained by translating the graph of Figure~\ref{7flower}. For such a graph, the reduction process always falls into a non-reducible configuration. However, this tilling can be easily $(9,4)$-colored using only 9 4-subsets of $\{0,\ldots, 8\}$ by extending the precoloring of Figure~\ref{7flower}.

\begin{figure}[ht]
\centering
\includegraphics[width=8cm]{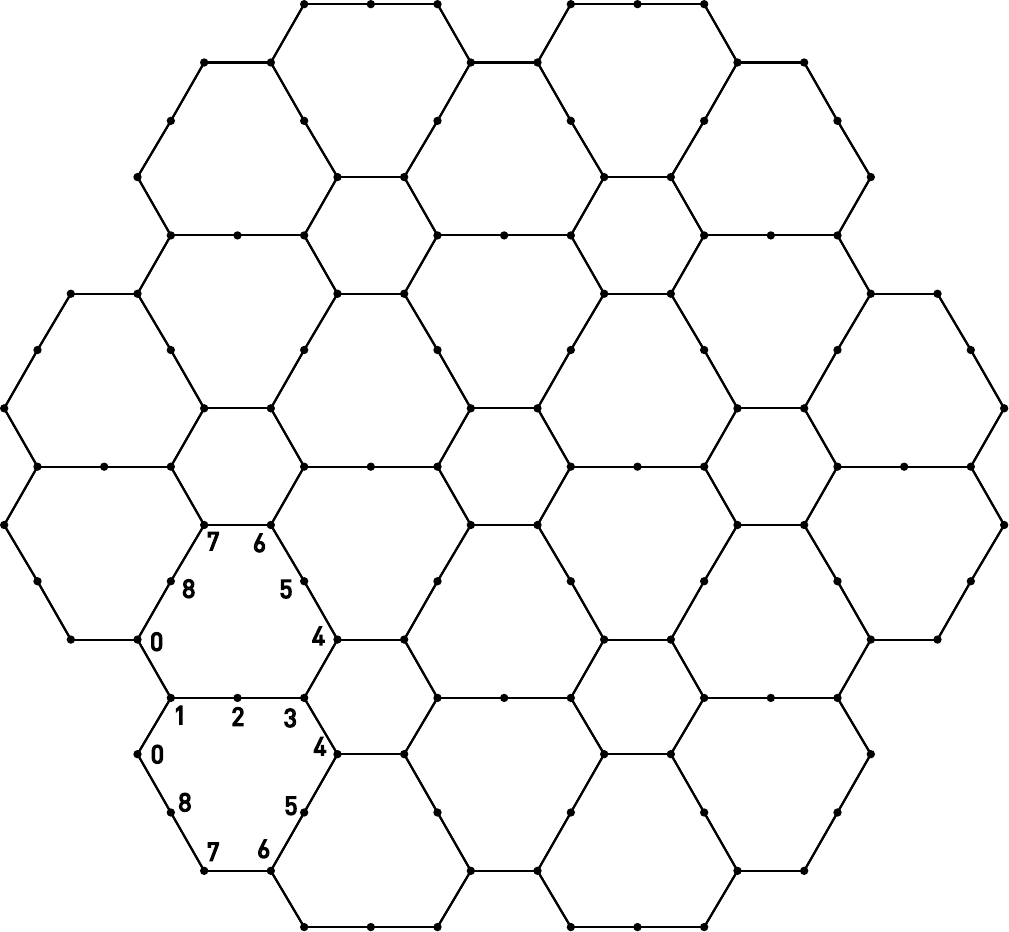}
\caption{\label{7flower}A hexagonal graph that contains none of our $(9,4)$-reducible configurations but can be $(9,4)$-colored by extending the coloring of the two 9-cycles (where number $i$ represents the 4-colors subset $[i,i+4 \bmod{9}]$) to the other 9-cycles.}
\end{figure}

Note that our reduction algorithm is very basic and could be improved in many ways, for instance by searching first for reducible configurations on its boundary instead of looking at each node in sequence.
%However there remain some configurations that are not reducible with these handles. %Table \ref{tb3} presents the percentage of graphs $G_{4,d}$ for which the core is not empty, depending on the distance $d$ between flowers and Figure~\ref{fig:exirred} presents two such examples. Notice that for distance $d=8$ or 9, no graph is reducible since the graph on the right side of Figure~\ref{fig:exirred} is always present.
% 
% \begin{figure}
%  \centering
% %\oli{Exemple à choisir}
% \includegraphics[width=4cm]{exIrred2}\hspace*{1.5cm}\includegraphics[width=6cm]{exIrred1}
% \caption{\label{fig:exirred}Two hexagonal graphs that do not contain any handle from $\mF_{9,4}\cup\mF'_{9,4}$.}
% \end{figure}

\subsection{Possible Extensions}
Even if we were not able to completely solve Conjecture~\ref{mcdconj} by hand, we made a progress by providing many new reducible configurations that may be used to obtain a proof of the conjecture in a similar way as for proving such graphs are $(7,3)$-colorable, i.e., showing that in any hexagonal graph, there always exists one of these $(9,4)$-reducible configurations somewhere in (the periphery of) the graph. Maybe these new configurations are not enough for such a proof and one will have to find other reducible configurations. In particular, we believe that there are many $(9,4)$-reducible configurations with a cycle and 3 or 4 paths of prescribed lengths going outside the cycle, other than the 25 configurations of Figure~\ref{fig:newhandles} but we did not try to find them exhaustively.

Another possible direction is to restrict to subclasses of hexagonal graphs, for instance we can ask: is any hexagonal graph with odd-girth 11 $(9,4)$-colorable? or even $(11,5)$-colorable?
We thus propose the following conjecture that generalizes Conjecture~\ref{mcdconj}:
\begin{conjecture}
For any $k\ge 1$, any hexagonal graph $G$ of odd-girth at least $2k+1$ is $(2k+1,k)$-colorable.
\end{conjecture}
The conjecture is true for $k\le 3$ and for $k=4$ for hexagonal graphs $G$ such that $\core_{\mF_{9,4}\cup \mF'_{9,4}\cup \mF''_{9,4}}(G)=\emptyset$.
A more general (but weaker) result in this direction is the one of Klostermeyer and Zhang~\cite{KZ02} proving that any planar graph with odd girth at least $10k-7$ (and $k\ge 2$) is $(2k+1,k)$-colorable.

The three-dimensional generalization (called cannonball graph~\cite{SZ2016}) also seems interesting. One question could be: Is any triangle-free subgraph of the 3D triangular grid $(5,2)$-colorable?

\section{Path multicoloring}
In this section we present general results about $(a,b)$-colorings of a path with precolored end-vertices. As a corollary, we will obtain Theorem~\ref{ThmP}.

Recall that $a,b,e$ are three integers such that $a=2b+e$. 
We say that $X$ is a {\em good} set if $\vert X \vert = b $ and $X \subset \{1,\ldots,a\}$. 

\begin{definition}
 Let $X,Y$ be good sets and $k$ be an integer. The sets $X$ and $Y$ are said to be {\em $k$-compatible} if :
$$\vert X \cap Y \vert \geq b - e \frac{k}{2} \ , \ if \ k \ is \ even; $$
$$\vert X \cap Y \vert \leq  e \frac{k-1}{2} \  , \ if \ k \ is \ odd. $$ 
Moreover, if the above two inequalities are equalities, we say that $X$ and $Y$ are {\em $k$-exactly-compatible}.
\end{definition}

The following result is already proven in~\cite[Lemma 4]{jan00} in a different way.
\begin{theorem}
\label{theoremprincipal}
For any integer $n\ge 0$, if $\varphi$ is a $(a,b)$-coloring of the path $P_{n+1}$, then 
$\varphi(0)$  and $\varphi(n)$ are $n$-compatible.
\end{theorem}

\begin{proof}
Let $\varphi$ be a $(a,b)$-coloring of the path $P_{n+1}$. The assertion trivially holds for $n=0$. If $n=1$ then since $v_0$ and $v_1$ are neighbors, we have $\varphi(0) \cap \varphi(1) = \emptyset$ then $\vert \varphi(0) \cap \varphi(1) \vert =0 \leq 0$. If $n=2$ then we have $\varphi(0) \cap \varphi(1) = \varphi(1) \cap \varphi(2) = \emptyset$, so $b+e = a-b = \vert \{1,\ldots,a\} \setminus \varphi(1) \vert \geq \vert \varphi(0) \cup \varphi(2) \vert = 2b - \vert \varphi(0) \cap \varphi(2) \vert $, then $\vert \varphi(0) \cap \varphi(2) \vert \geq b-e $.

Assume that the property is true for $n \geq 2$ even. Let $\varphi$ be a $(a,b)$-coloring of $P^{n+2}$, then $\varphi(0)$ and $\varphi(n)$ are $n$-compatible, i.e. $ \vert \varphi(0) \cap \varphi(n) \vert \geq b - e \frac{n}{2}$ and $\varphi(n)$ and $\varphi(n+2)$ are $2$-compatible, i.e., $ \vert \varphi(n) \cap \varphi(n+2) \vert \geq b - e $. So $\vert \varphi(0) \cap \varphi(n+2) \vert \geq \vert \varphi(0) \cap \varphi(n) \cap \varphi(n+2) \vert  \geq \vert \varphi(0) \cap \varphi(n) \vert - \vert ( \varphi(0) \cap \varphi(n)) \setminus \varphi(n+2)  \vert \geq (b - e\frac{n}{2}) - e =b - e\frac{n+2}{2}$.

Assume now that the property is true for $n \geq 1$ odd. Let $\varphi$ be a $(a,b)$-coloring of $P^{n+2}$, then since $\varphi(0)$, $\varphi(n)$ are $n$-compatible and $\varphi(n)$, $\varphi(n+2)$ are $2$-compatible, we have $ \vert \varphi(0) \cap \varphi(n) \vert \leq  e \frac{n-1}{2}$, and  $ \vert \varphi(n) \cap \varphi(n+2) \vert \geq b - e $. So $\vert \varphi(0) \cap \varphi(n+2) \vert = \vert \varphi(0) \cap \varphi(n) \cap \varphi(n+2) \vert + \vert (\varphi(0) \cap \varphi(n+2)) \setminus \varphi(n) \vert \leq \vert \varphi(0) \cap \varphi(n) \vert + \vert \varphi(n+2) \setminus \varphi(n)  \vert \leq e \frac{n-1}{2} + e \leq e \frac{(n+2)-1}{2} $.
\end{proof}

For any two good sets $C_0 , C_n$, we define the canonical decomposition of subsets: 
$$C(1)=C_0 \cap C_n \ , \ C(2)=C_0 - C_n \ , \  C(3)=C_n - C_0 \ , \ C(4)=\{1,\ldots,a\} - C_0 - C_n.$$

For an $(a,b)$-coloring $\varphi$ of the path $P_{n+1}$, we define $\forall j \in \{1,2,3,4\}$ and $\forall k \in \{0,\ldots,n\}$ : 
$$c(j,k) = \vert \varphi(k) \cap C(j) \vert,$$
$$ c(k) = \big( c(1,k),c(2,k),c(3,k),c(4,k) \big). $$

Therefore $\forall j \in \{1,2,3,4\}$ and $\forall k \in \{0,\ldots,n-1\}$, we have the property:
$$ c(j,k) + c(j,k+1) \leq \vert C(j) \vert. $$ 

\begin{proposition}
\label{propositionexactly}
If $\varphi$ is a $(a,b)$-coloring of the path $P_{n+1}$ such that $n=2m$ is an even integer and $b\ge em$ and $\varphi(0)$ and $\varphi(n)$ are $n$-exactly-compatible, then $\forall k \in \{0,\ldots,n\}$:

$$ c(k) = \left \{ \begin{array}{ll}
(b-em,e(m-s),es,0), & \text{if } k=2s \text{ is even};\\
(0, es, e(m-s-1) , b - e(m-1)), & \text{if } k=2s+1 \text{ is odd.} 
\end{array}\right.$$
\end{proposition}

\begin{proof}
Let $n=2m$ be an even integer such that $b \geq em$ and let $C_0 , C_n$ be two $n$-exactly-compatible good sets. Then we have by construction $\vert C(1) \vert = b-em$, $ \vert C(2) \vert = \vert C(3) \vert = em$, and $\vert C(4) \vert =a- \vert C(1) \vert - \vert C(2) \vert - \vert C(3) \vert =  b-e(m-1) $.

If $\varphi$ is a $(a,b)$-coloring of the path $P_{n+1}$ such that $\varphi(0)=C_0$ and $\varphi(n)=C_n$, then $\forall k \in \{0,\ldots,n\}$:
\begin{itemize}
 \item If $k=2s$ is even, then by Theorem \ref{theoremprincipal},  $\varphi(k)$ and $\varphi(n)$ are $(n-k)$-compatible and thus $\vert \varphi(k) \cap \varphi(n) \vert \geq b - e \frac{n-k}{2} = b - e(m-s) $; $\varphi(k)$ and $\varphi(0)$ are $k$-compatible and thus $\vert \varphi(k) \cap \varphi(0) \vert \geq b - e \frac{k}{2} = b - es$. 
So $\vert \varphi(k) \setminus \varphi(0) \vert  \leq es$ and thus $c(3,k) \leq es $. Since $\vert C(1) \vert = b-em $ and $\vert \varphi(k) \cap \varphi(n) \vert \geq b - e(m-s) $ then $c(3,k) \geq es$, so $c(3,k) = es$.  Since $c(1,k)+c(2,k) = \vert \varphi(k) \cap \varphi(0) \vert \geq  b - es $, and $c(1,k)+c(2,k)+c(3,k)+c(4,k)= b$, then $c(4,k)=0$ and $c(1,k)+c(2,k) =  b - es $. Since $c(1,k)+c(3,k) =\vert \varphi(k) \cap \varphi(n) \vert \geq  b - e(m-s) $, we have $c(1,k) \geq b-em = \vert C(1) \vert$, and thus $c(1,k)=b-em$ and $c(2,k)=e(m-s)$.
\item If $k=2s+1$ is odd, by the above, we have $c(2s)=(b-em,e(m-s),es,0)$ and $c(2s+2)=(b-em,e(m-s-1),e(s+1),0)$. Since $\vert C(1) \vert = b-em=c(1,2s)$ then $c(1,k)=0$. Since $\vert C(2) \vert = em$ and $c(2,2s)=e(m-s)$ then $c(2,k) \leq es$. Since $\vert C(3) \vert = em$ and $c(3,2s+2)=e(s+1)$ then $c(3,k) \leq e(m-s-1)$. Since $\vert C(4) \vert = b-e(m-1)$  then $c(4,k) \leq b-e(m-1)$. But $c(1,k)+c(2,k)+c(3,k)+c(4,k)= b$, therefore all the inequalities are equalities.
\end{itemize}
\end{proof}

\begin{proposition}
\label{propositionexactlynimpair}
If $\varphi$ is a $(2b+1,b)$-coloring of the path $P_{n+1}$ such that $n=2m+1$ is an odd integer and $b\ge em$ and $\varphi(0)$ and $\varphi(n)$ are $n$-exactly-compatible, then $\forall k \in \{0,\ldots,n\}$:

$$ c(k) = \left \{ \begin{array}{lllll}
(e(m-s),& b-em,& 0,& es), & \text{if } k=2s \text{ is even};\\
(es,& 0,& b-em, & e(m-s)), & \text{if } k=2s+1 \text{ is odd.} 
\end{array}\right.$$
\end{proposition}

\begin{proof}
Let $n=2m+1$ be an odd integer such that $b \geq em$ and let $C_0 , C_n$ be two $n$-exactly-compatible good sets. Then we have by construction $ \vert C(1) \vert = em$, since $\vert C_0 \vert = \vert C_n \vert = b$ then $ \vert C(2) \vert = \vert C(3) \vert = b-em$ , since $\vert C(1) \vert + \vert C(2) \vert + \vert C(3) \vert + \vert C(4) \vert =a$ then $\vert C(4) \vert = e(m+1) $.\\
\begin{itemize}
 \item If $\varphi$ is a $(a,b)$-coloring of the path $P_{n+1}$ such that $\varphi(0)=C_0$ and $\varphi(n)=C_n$, then $\forall k \in \{0,\ldots,n\}$ :\\
If $k=2s$ is even, then by Theorem \ref{theoremprincipal}, then  $\varphi(k)$ and $\varphi(n)$ are $(n-k)$-compatible and thus $\vert \varphi(k) \cap \varphi(n) \vert \leq  e \frac{n-k-1}{2} = e(m-s) $; $\varphi(k)$ and $\varphi(0)$  are $k$-compatible and thus $\vert \varphi(k) \cap \varphi(0) \vert \geq b - e \frac{k}{2} = b - es $. We have $c(1,k)=\vert \varphi(0) \cap \varphi(k) \vert - c(2,k) \geq b-es - \vert \varphi(2) \vert=e(m-s)$. Since  $\vert \varphi(k) \cap \varphi(n) \vert \leq  e(m-s) $ we have $c(1,k) \leq e(m-s)$, then $c(1,k) = e(m-s)$. Therefore $c(2,k) = b-em$ and $c(3,k)=0$. Since $c(1,k)+c(2,k)+c(3,k)+c(4,k)= b$, then $c(4,k)=es$.\\
\item If $k=2s+1$ is odd, by the above properties, we have $c(2s)=(e(m-s),b-em,0,es)$ and $c(2s+2)=(e(m-s-1),b-em,0,e(s+1))$. Since $\vert C(2) \vert = b-em$, then $c(2,k)=0$. Since $\vert C(1) \vert = em$, then $c(1,k) \leq es$. Since $\vert C(4) \vert = e(m+1)$, then $c(4,k) \leq e(m-s)$. Since $c(1,k)+c(2,k)+c(3,k)+c(4,k)= b$, then $c(3,k) \geq b-em = \vert C(3) \vert$, therefore all the inequalities are equalities.
\end{itemize}

\end{proof}

In order to obtain a reciprocal result of Theorem~\ref{theoremprincipal}, we define, for any ordered set $I=\{x_1, \dots , x_f \}$ (with $x_1 < x_2 \dots < x_f$),
 $\first (k,I)= \{ x_1 , \dots , x_k \}$, and  $\last(k,I)= \{ x_{f-k+1} , \dots , x_f \}$. We have the following easy useful fact:
$$\first(k,I) \cap \last(k',I) = \emptyset \iff k+k' \leq \vert I \vert. $$

\begin{lemma}
\label{lemmaalgorithmsatifiesexactly}
If $C_0,C_n$ are good sets such that $C_0$ and $C_n$ are $n$-exactly-compatible, then  a $(a,b)$-coloring $\varphi$ of the path $P_{n+1}$ such that $\varphi(0)=C_0$ and $\varphi(n)=C_n$ can be computed in linear time.
\end{lemma}

\begin{proof}
Let $C_0,C_n$ be good $n$-exactly-compatible sets.

If $n=2m$ is even, then by Proposition \ref{propositionexactly}, $\forall k \in \{0,\ldots,n\}$  we set:
$$ \varphi(k)=\left \{ \begin{array}{ll}
             C(1) \cup \first(e(m-s), C(2) ) \cup  \first( es , C(3)) & \text{if } \ k=2s \text{ is even};\\
	    \last( es, C(2) ) \cup \last( e(m-s-1) ,C(3) ) \cup C(4) & \text{if } k=2s+1 \text{ is odd.}
                       \end{array}\right.$$
%Therefore this linear time algorithm computes  $C$ a -\intcoloring of a path $P_{n+1}$ such that $\varphi(0)=C_0$ and $\varphi(n)=C_n$.

If $n=2m+1$ is odd, then by Proposition \ref{propositionexactlynimpair}, $\forall k \in \{0,\ldots,n\}$  we set:
$$ \varphi(k)=\left \{ \begin{array}{ll}
            \first( e(m-s) , C(1) ) \cup C(2) \cup \first( es , C(4) )  & \text{if } \ k=2s \text{ is even};\\
	    \last( es , \varphi(1) ) \cup C(3) \cup \last( e(m-s) , C(4) ) & \text{if } k=2s+1 \text{ is odd.}
                       \end{array}\right.$$

Therefore we obtain in both cases in linear time a $(a,b)$-coloring $\varphi$ of $P_{n+1}$ such that $\varphi(0)=C_0$ and $\varphi(n)=C_n$.
\end{proof}

\begin{lemma}
\label{lemmaalgorithmfacilepair}
If $C_0,C_{2n}$ are good sets such that $\vert C_0 \cap C_{2n} \vert \geq b-e$, then a $(a,b)$-coloring $\varphi$ of the path $P_{2n+1}$ such that $\varphi(0)=C_0$ and $\varphi(2n)=C_{2n}$ can be computed in linear time.
\end{lemma}

\begin{proof}
Let $C_0,C_{2n}$ be good sets such that $\vert C_0 \cap C_{2n} \vert \geq b-e$, then $\vert C_0 \cup C_{2n} \vert \leq 2b - (b-e) = b +e $. Hence there exists a good set $X \subset \{1,\ldots,a\} - C_0 - C_{2n}$. We set $\varphi(0)=C_0$, and $\forall k \in \{1,\ldots,2n \}$, $\varphi(k)=C_{2n}$ if $k$ is even, and $\varphi(k)=X$ otherwise. Therefore this linear time algorithm computes a $(2b+e,b)$-coloring $\varphi$ of the path $P_{2n+1}$ such that $\varphi(0)=C_0$ and $\varphi(2n)=C_{2n}$.
\end{proof}

These lemmas allow to prove the next result which is central in this paper as it will be used several times in the next sections.
\begin{theorem}
\label{theoremalgorithmcoloration}
If $C_0,C_n$ are $n$-compatible sets, then a $(a,b)$-coloring $\varphi$ of the path $P_{n+1}$ such that $\varphi(0)=C_0$ and $\varphi(n)=C_n$ can be computed in linear time.
\end{theorem}

\begin{proof}
Let $C_0,C_n$ be good $n$-compatible sets.

If $n$ is even, let $k=2s \in \{0,\ldots,n\}$ be an even integer such that $b-es \leq \vert C_0 \cap C_n \vert < b-es + e$. Let $X \subset C_0 \cap C_n$ such that $\vert X \vert = b-es$ ; and $Y=C_n - C_0$, then $es \geq \vert Y \vert >es - e$. Let $Z \subset \{1,\ldots,a\} - C_0 - C_n$ such that $\vert Z \vert = es - \vert Y \vert$. We choose $C_k=X \cup Y \cup Z$, then $C_k$ is a good set and $\vert C_0 \cap C_k \vert = b-es$, therefore by Lemma \ref{lemmaalgorithmsatifiesexactly}, a $(a,b)$-coloring $\varphi$ of the path $P^k$ such that $\varphi(0)=C_0$ and $\varphi(k)=C_k$ can be computed in linear time. By construction $\vert C_k \cap C_n \vert > b-e$ and $n-k$ is even, then by Lemma \ref{lemmaalgorithmfacilepair}, we can complete the coloring for the rest of the path (between $v_k$ and $v_n$) in linear time. 

If $n$ is odd, let $k=2s+1 \in \{0,\ldots,n\}$ be an odd integer such that $es \leq \vert C_0 \cap C_n \vert < es + e$. Let $X \subset C_0 \cap C_n$ such that $\vert X \vert = es$; and $Y \subset C_n - C_0$ such that  $ \vert Y \vert =b- es - e$. Let $Z \subset \{1,\ldots,a\} - C_0 - C_n$ such that $\vert Z \vert = e$. We choose $C_k=X \cup Y \cup Z$, then $C_k$ is a good set and $\vert C_0 \cap C_k \vert = es$, therefore by Lemma \ref{lemmaalgorithmsatifiesexactly}, a $(a,b)$-coloring $\varphi$ of a path $P^k$ such that $\varphi(0)=C_0$ and $\varphi(k)=C_k$ can be computed in linear time. By construction $\vert C_k \cap C_n \vert = b-e$ and $n-k$ is even, then by  Lemma \ref{lemmaalgorithmfacilepair}, we can complete the coloring for the rest of the path (between $v_k$ and $v_n$) in linear time. 
\end{proof}

If $C_0,C_n$ are good sets such that $n \geq \even( \frac{2b}{e} ) $, then they are $n$-compatible and therefore we have the following result:

\begin{corollary}
\label{corollarytheoremreciproquealgorithmnsupeven2bsure}
If $C_0,C_n$ are good sets such that $n \geq \even(\frac{2b}{e}) $, then a $(a,b)$-coloring $\varphi$ of the path $P_{n+1}$ such that $\varphi(0)=C_0$ and $\varphi(n)=C_n$ can be computed in linear time.
\end{corollary}

\begin{proposition}
If a graph $G$ contains a parity handle $H$ and $G-\int(H)$ is $(a,b)$-colorable, then $G$ is $(a,b)$-colorable.
\end{proposition}

\begin{proof}
Let $H=(v_0,\ldots,v_n)$ be a parity handle of length $n$ in $G$ such that there exists a $(a,b)$-coloring $\varphi$ of $G - \int(H)$.
Then, by definition, there exists another path $P'=(v'_0=v_0, v'_1,\ldots, v'_{n'}=v_n)$ in $G- \int(H)$ with $n\ge n'$ and $n$ and $n'$ having the same parity.
We extend the coloring $\varphi$ by setting $\forall k \in \{0,\ldots,n'\}, \varphi(v_k)=\varphi(v'_k)$, and $\forall k \in \{n'+1,\ldots,n\}$, $\varphi(v_k)=\varphi(v'_{n'})$ if $k-n'$ is even and $\varphi(v_k)=\varphi(v'_{n'-1})$ otherwise. %Therefore the $(a,b)$-coloring $\varphi$ has been extended to $G$.
\end{proof}

This proposition along with Corollary~\ref{corollarytheoremreciproquealgorithmnsupeven2bsure} allows to prove Theorem~\ref{ThmP}.

\section{Properties of $(2b+1,b)$-colorings of a path}
This section contains a series of lemmas presenting properties on the color sets of some interior vertices of a path with a given $(2b+1,b)$-coloring. These lemmas will be used in Section~\ref{sec3} to reduce the number of cases to consider when proving the reducibility of some H-handles.

Remember that the vertices of a path $P_{n+1}$ are denoted by $v_0, v_1,\ldots, v_n$ and that for any two good sets $C_0 , C_n$, 
$C(1)=C_0 \cap C_n \ , \ C(2)=C_0 - C_n \ , \  C(3)=C_n - C_0 \ , \ C(4)=\{1,\ldots,a\} - C_0 - C_n$.

\begin{lemma}
\label{lemmaP4}
For any  $(2b+1,b)$-coloring $\varphi$ of $P_{5}$, there exist four distinct good sets $X_1$, $X_2$, $X_3$ and $X_4$, such that
\begin{description}
 \item[(i)] $\vert X_1 \cap X_2 \vert = \vert X_1 \cap X_3 \vert =b-1$, $\vert X_1 \cap X_4 \vert = b-2$, and
\item[(ii)] for each $i\in\{1,2,3,4\}$, there exists a $(2b+1,b)$-coloring $\varphi'$ of $P_{5}$ such that $\varphi'(v_0)=\varphi(v_0)$,  $\varphi'(v_4)=\varphi(v_4)$ and $ \varphi'(v_2) =X_i$.
\end{description}
%$\vert X_1 \cap X_2 \vert = \vert X_1 \cap X_3 \vert =b-1$, $\vert X_1 \cap X_4 \vert = b-2$ and for each $i\in\{1,2,3,4\}$, there exists a $(2b+1,b)$-coloring $\varphi'$ of $P_{5}$ such that $\varphi'(v_0)=\varphi(v_0)$,  $\varphi'(v_4)=\varphi(v_4)$ and $ \varphi'(v_2) =X_i$. 
\end{lemma}

\begin{proof}
Let $\varphi$ be a $(2b+1,b)$-coloring of $P_{5}$ and let $C_0=\varphi(v_0)$ and $C_4=\varphi(v_4)$. By Theorem~\ref{theoremprincipal}, $|C(1)|=|C_0 \cap C_4|\ge b-2$. Thus we have three cases to consider:
\begin{description}
\item[Case 1:] $\vert C(1) \vert = b$. We choose any $Y \subset C(1)$, any $Y_1 \subset C_0 - Y$, $Y_2=C_0 - Y - Y_1$, any $Z,Z' \subset \{1,\ldots,a\} - C_0$, such that $\vert Y \vert =b-2$, and $\vert Z \vert =\vert Z' \vert =\vert Y_1 \vert =\vert Y_2 \vert =1$. Then we set $X_1=Y \cup Y_1 \cup Z$, $X_2=C_0$, $X_3=Y \cup Y_1 \cup Z'$, and $X_4=Y \cup Y_2 \cup Z'$.
\item[Case 2:] $\vert C(1) \vert = b-1$. We choose any $Y \subset C(1)$, and any $Z,Z' \subset \{1,\ldots,a\} - C_0 - C_4$, such that $\vert Y \vert =b-2$, and $\vert Z \vert  =\vert Z' \vert  =1$, and we set $X_1=C(1) \cup Z$, $X_2=C_0$, $X_3=C(1) \cup Z'$ and $X_4=Y \cup (C_0 - C(1)) \cup (C_4 - C(1))$.
\item[Case 3:] $\vert C(1) \vert = b-2$. Then we note $C(2)=Y_2 \cup Y'_2$ and $C(3)=Y_3 \cup Y'_3$, with $\vert Y_i |Y'_i|= \vert =1$, for $i=1,2$, and we set $X_1=D_1 \cup Y_2 \cup Y_3$, $X_2=C(1) \cup Y_2 \cup Y'_3$, $X_3=C(1) \cup Y'_2 \cup Y_3$, and $X_4=C(1) \cup Y'_2 \cup Y'_3 $.
\end{description}
For each case and each $i\in\{1,2,3,4\}$, the pairs $(C_0,X_i)$ and $(X_i, C_4)$ are both $2$-compatible and thus Theorem \ref{theoremalgorithmcoloration} allows to conclude.
\end{proof}

\begin{lemma}
\label{lemmaP5}
For any  $(2b+1,b)$-coloring $\varphi$ of $P_6$, there exist  two distinct sets $X$ and $Y$ such that $\vert X \vert =\vert Y \vert = 2$ and for any $x\in X$ and $y\in Y$ there exists a $(2b+1,b)$-coloring $\varphi'$ of $P^5$ such that $\varphi'(v_0)=\varphi(v_0)$,  $\varphi'(v_5)=\varphi(v_5)$ and $\{x,y\} \subset\varphi'(v_2)$. 
\end{lemma}

\begin{proof}
Let $\varphi$ be a $(2b+1,b)$-coloring of $P_6$ and let $C_0=\varphi(v_0)$ and $C_5=\varphi(v_5)$. By Theorem~\ref{theoremprincipal}, $\vert C(1) \vert \leq 2$ and thus $\vert C(2) \vert \geq b-2$.
\begin{description}
\item[Case 1:] $\vert C(1) \vert = 0$. We choose any $I \subset C_0$, $X=C_0 - I$ and any $ Y \subset \{1,\ldots,a\} - C_0$, such that $\vert I \vert =b-2$, and $\vert Y \vert =2$. %, for any $x\in X$ and $y\in Y$. %, and we set $\varphi'(v_2)=I \cup \{x,y\}$.
\item[Case 2:] $\vert C(1) \vert \geq 1$. We choose any $I \subset  C(2)$, $X=C_0 - I$ and $Y \subset C(4)$ such that $\vert I \vert =b-2$ and $\vert Y \vert =2$.%, and we set $\varphi'(v_2)=I \cup \{x,y\}$.
\end{description}
For each case and each $x\in X, y\in Y$, by construction, $C_0, I\cup\{x,y\}$ are 2-compatible and $I\cup\{x,y\}, C_5$ are 3-compatible, hence Theorem \ref{theoremalgorithmcoloration} allows to conclude.
\end{proof}

\begin{lemma}
\label{lemmaP2n+1un}
For any $(2b+1,b)$-coloring $\varphi$ of $P_{2n+2}$ with $n \leq b - 1$, there exist  two distinct sets $X$ and $Y$, such that $\vert X \vert  =b-n$ and $\vert Y \vert = n+1$, and for any $Y' \subset Y$ such that $\vert Y' \vert = n$, there exists a $(2b+1,b)$-coloring $\varphi'$ of $P_{2n+2}$ such that $\varphi'(v_0)=\varphi(v_0)$,  $\varphi'(v_{2n+1})=\varphi(v_{2n+1})$ and $ \varphi'(v_1) = X \cup Y'$. 
\end{lemma}

\begin{proof}
Let $n \leq b - 1$ and $\varphi$ be a $(2b+1,b)$-coloring of $P_{2n+2}$ and let $C_0=\varphi(v_0)$ and $C_{2n+1}=\varphi(v_{2n+1})$. We have  $\vert C(1) \vert \leq n$ and thus $\vert C(3) \vert \geq b-n$. We choose any $X \subset C(3)$ with $\vert X \vert =b-n$ and $Y= \{1,\ldots,a\} - X -C_0 $. For any $Y' \subset Y$ such that $\vert Y' \vert = n$, we let $C_1= X \cup Y'$. Then $\vert C_0 \cap C_1 \vert = 0$ and $\vert C_{2n+1} \cap C_1 \vert \geq b-n$, therefore Theorem \ref{theoremalgorithmcoloration} allows to conclude.
\end{proof}

\begin{lemma}
\label{lemmaP4nnmilieu}
For any $(2b+1,b)$-coloring $\varphi$ of $P_{4n+1}$ with $2n\leq b$, there exists a set $X \subset \{1,\ldots,a\}$  such that $\vert X \vert = b+2n$, and for any $X' \subset X $ with $\vert X' \vert = n$,  there exists a $(2b+1,b)$-coloring $\varphi'$ of $P_{4n+1}$ such that $\varphi'(v_0)=\varphi(v_0)$,  $\varphi'(v_{4n})=\varphi(v_{4n})$ and $ X' \subset \varphi'(v_{2n})$. 
\end{lemma}

\begin{proof}
Let $2n \leq b$ and $\varphi$ be a $(2b+1,b)$-coloring of  $P_{4n+1}$. Let $C_0=\varphi(0)$ and $C_{4n}=\varphi(4n)$.

\begin{description}
\item[Case 1:]  $\vert C(1) \vert  \geq b-n$. We choose any set $X \subset \{1,\ldots,a\}$  with $\vert X \vert = b+2n$,  and for any $X' \subset X $ such that $\vert X' \vert = n$, we choose any $I \subset C(1)$ such that $\vert I \vert = b-n$,  and we choose any $Y \subset C(4)$ such that  $\vert Y \vert =  \vert I \cap X' \vert $. We then set $C_{2n}=I \cup X' \cup Y$.

\item[Case 2:] $ \vert C(1) \vert  \leq b-n-1$.  Let $i= b-n- \vert C(1) \vert $. We choose any set $X$  with $C_0 \cup C_{4n} \subset X \subset \{1,\ldots,a\}$  and $\vert X \vert = b+2n$,  and for any $X' \subset X $ such that $\vert X' \vert = n$, we choose any $Y_2$ with $ X' \cap C(2) \subset Y_2 \subset C(2)$ and  $\vert Y_2 \vert =  \max( i , \vert  C(2) \cap X' \vert )$, and we choose any $Y_3$ with $X' \cap C(3) \subset Y_3 \subset C(3) $ and  $\vert Y_3 \vert =  \max( i , \vert  C(3) \cap X' \vert )$, and we choose any $Z$ with $X' \cap C(4) \subset Z \subset X - C(1) - Y_2 -Y_3 $ and  $\vert Z \vert = b- \vert  C(1)  \vert  -  \vert  Y_2  \vert  -  \vert  Y_3  \vert$. Such choice is always possible since $ \vert X' \cap C(4) \vert \leq \vert X - C_0 - C_{4n} \vert = b+2n - (b+n+i) = n-i$. We then set $C_{2n}=C(1) \cup Y_2 \cup Y_3 \cup  Z $.
\end{description}
In both cases, the pairs $(C_0,C_{2n})$ and $(C_{2n},C_{4n})$ are both $2n$-compatible and thus Theorem \ref{theoremalgorithmcoloration} allows to conclude.
\end{proof}

\begin{lemma}
\label{lemmaP4n+2nmilieu}
For any $(2b+1,b)$-coloring $\varphi$ of $P_{4n+3}$ with $3 \leq 2n+1 \leq b$, there exists a set $X \subset \{1,\ldots,a\}$  with $\vert X \vert = b+2n+1$, and for any $X' \subset X $ such that $\vert X' \vert = n+1$,  there exists a $(2b+1,b)$-coloring $\varphi'$ of $P_{4n+3}$ such that $\varphi'(v_0)=\varphi(v_0)$,  $\varphi'(v_{4n+2})=\varphi(v_{4n+2})$ and $X' \cap \varphi'(v_{2n+1})  = \emptyset$. 
\end{lemma}

\begin{proof}
Let $\varphi$ be a $(2b+1,b)$-coloring of $P_{4n+3}$ and $3 \leq 2n+1 \leq b$. Let $C_0=\varphi(v_0)$ and $C_{4n+2}=\varphi(v_{4n+2})$. 

\begin{description}
\item[Case 1:] $\vert C(1) \vert  \geq b-n$. Let $i= b-\vert C(1) \vert$. We choose any set $ X \subset \{1,\ldots,a\}$  such that $\vert X \vert = b+2n+1$.  For any $X' \subset X $ such that $\vert X' \vert = n+1$, if $\vert C(4) - X' \vert \geq b$ then  we choose any good set $I \subset C(4) - X' $ and we set $C_{2n+1}=I$; otherwise we choose $I = C(4) - X'$, so $\vert I \vert = b-i  +1 - \vert X' \cap C(4) \vert$. If $\vert I \vert \geq b-n$, then we choose any $Y \subset C_0 - X' $ such that $\vert Y \vert = b - \vert I \vert$ and we set $C_{2n+1}=I \cup Y$.  Else ($\vert I \vert <  b-n$), we let $Y_2=C(2) - X'$. If $\vert  C(3) - X' \vert \geq b - \vert I \vert - \vert Y_2 \vert $, then we choose $Y_3 \subset C(3) - X'$ such that $\vert Y_3 \vert = b - \vert I \vert - \vert Y_2 \vert $ and we set $C_{2n+1}=I \cup Y_2 \cup Y_3$. Otherwise, $\vert  C(3) - X' \vert < b - \vert I \vert - \vert Y_2 \vert $, and then we choose $Y_3 = C(3) - X'$. By construction we have $\vert I \vert + \vert Y_2 \vert +\vert Y_3 \vert = \vert C(4) \vert + \vert C(2) \vert +\vert C(3) \vert - ( n+1) + \vert X' \cap C(1) \vert \geq b +i-n $, then we  choose any $Z \subset C(1) - X' $ such that $\vert Z \vert = b - \vert I \vert - \vert Y_2 \vert - \vert Y_3 \vert $ and we set $C_{2n+1}=I \cup Y_2 \cup Y_3 \cup Z$.

\item[Case 2 :] $ \vert C(1) \vert  \leq b-n-1$.  Let $i= b-n- \vert C(1) \vert $. We choose any set $ C_0 \cup C_{4n} \subset X \subset \{1,\ldots,a\}$  such that $\vert X \vert = b+2n+1$,  and for any $X' \subset X $ such that $\vert X' \vert = n+1$, we choose $I = C(4) - X' $. If $\vert C(2) - X' \vert \geq n $, then we choose $Y_2 \subset C(2) - X'$  such that $\vert Y_2 \vert = n$. If $\vert C(3) - X' \vert \geq b - \vert I \vert - \vert Y_2 \vert  $ then we choose $Y_3 \subset C(3) - X'$  such that $\vert Y_3 \vert = b - \vert I \vert - \vert Y_2 \vert   $ and we set $C_{2n+1}=I \cup Y_2 \cup Y_3 $. Otherwise, $\vert C(3) - X' \vert < b - \vert I \vert - \vert Y_2 \vert  $, so $b-n=\vert C(3) \vert + \vert C(4) \vert - \vert X' \vert < b- \vert Y_2 \vert = b-n$, a contradiction. We do the same with $\vert C(3) - X' \vert \geq n $.   Else $\vert C(2) - X' \vert < n $ and $\vert C(3) - X' \vert < n $, then we choose $Y_2 = C(2) - X'$.  If $\vert C(3) - X' \vert \geq b - \vert I \vert - \vert Y_2 \vert $, then we choose $Y_3 \subset C(3) - X'$  such that $\vert Y_3 \vert = b - \vert I \vert - \vert Y_2 \vert$ and we set $C_{2n+1}=I \cup Y_2 \cup Y_3 $. Otherwise, $\vert C(3) - X' \vert < b - \vert I \vert - \vert Y_2 \vert  $, so $b+1 \leq b+\vert C(3) \vert - n = \vert C(3) \vert + \vert C(4) \vert +\vert C(2) \vert - \vert X' \vert < b$, a contradiction.
\end{description}
Therefore by Theorem \ref{theoremalgorithmcoloration}, the lemma is proved. 
\end{proof}

\section{S-handle and H-handle reductions}
\label{sec3}
The purpose of this section is to prove Theorems~\ref{ThmS} and \ref{ThmH} of Section \ref{sec2}.
%The spider graph $S(n_1,n_2,n_3)$ is the graph obtained from a star of order 4 by replacing one of the three edges by a path of length $n_1$, another edge by a path of length $n_2$ and the third edge by a path of length $n_3$.
For an S-handle $S(n,n_1,n_2)$, let $P=h_0,h_1,\cdots,h_n$ be the central path and let $v_1$ and $v_2$ be the end-vertices of the paths of lengths $n_1$ and $n_2$, respectively. Similarly, for a H-handle $H(n_1,n_2,n,n_3,n_4)$, $P=h_0,h_1,\cdots,h_n$ is the central path and $v_1$ and $v_2$, $v_3$ and $v_4$ are the end-vertices of the paths of lengths $n_1$, $n_2$, $n_3$ and $n_4$, respectively. See Figure~\ref{fig:nota} for an illustration.

\begin{figure}[ht]
\begin{center}

\begin{tikzpicture}[scale=0.7]
\node at (2,0) [circle,draw=black,scale=0.8](t1){$h_0$};
\node at (6,0) [circle,draw=black,scale=0.8](t2){$h_n$};
\node at (0.6,1) [circle,draw=black,scale=0.8](h1){$v_1$};
\node at (0.6,-1) [circle,draw=black,scale=0.8](h2){$v_2$};
%\node at (7.4,1) [circle,draw=black,scale=0.8](h3){$v_3$};
%\node at (7.4,-1) [circle,draw=black,scale=0.8](h4){$v_4$};
\draw  (t1) -- (h1); % node[midway, above] {2};
\draw  (h2)-- (t1); % node[midway, below] {2};
\draw [ultra thick] (t1) -- (t2);% node[midway, above] {4};
%\draw  (t2) -- (h3); % node[midway, above] {2};
%\draw  (h4)-- (t2); % node[midway, below] {2};

\node at (12,0) [circle,draw=black,scale=0.8](t1){$h_0$};
\node at (16,0) [circle,draw=black,scale=0.8](t2){$h_n$};
\node at (10.6,1) [circle,draw=black,scale=0.8](h1){$v_1$};
\node at (10.6,-1) [circle,draw=black,scale=0.8](h2){$v_2$};
\node at (17.4,1) [circle,draw=black,scale=0.8](h3){$v_3$};
\node at (17.4,-1) [circle,draw=black,scale=0.8](h4){$v_4$};
\draw  (t1) -- (h1); % node[midway, above] {2};
\draw  (h2)-- (t1); % node[midway, below] {2};
\draw [ultra thick] (t1) -- (t2);% node[midway, above] {4};
\draw  (t2) -- (h3); % node[midway, above] {2};
\draw  (h4)-- (t2); % node[midway, below] {2};

\end{tikzpicture}
\end{center}
\caption{\label{fig:nota}Notation for the vertices of S-handles (on the left) and H-handles (on the right).}
\end{figure}

\begin{proposition}
\label{proposition1handle}
For any $b,e$ with $b\ge e$ and any graph $G$, $S(\even(\frac{2b}{e}) - 1,2,1 )$ is a smallest $(2b+e,b)$-reducible S-handle in $G$.
\end{proposition}

\begin{proof}
Let $n=\even( \frac{2b}{e} )- 1$.

For the minimality, we present a counter-example showing that $S(\even(\frac{2b}{e}) - 1,1,1)$ is not reducible: if $C(v_1)=\{1,\ldots,b\}$, $C(v_2)=\{e+1,\ldots, b+e\}$, $C(h_n)=\{b+e+1,\ldots,2b+e\}$, then the color set of $h_0$ must be $\{b+e+1,\ldots,2b+e\}$ but then the color sets of $h_0$ and $h_n$ are not $n$-compatible, hence a contradiction. Also, the handle  $S(\even(\frac{2b}{e}) - 2,2,1 )$ is not reducible since in this case there is a path between $v_2$ and $h_n$ of length $\frac{2b}{e} - 2+1<\frac{2b}{e}$, hence the coloring cannot be extended to the interior vertices of the handle if these two vertices have the same color set.

In order to prove reducibility, let $G$ be a graph containing a handle $H=S(n,2,1)$ and $\varphi$ be a $(2b+e,b)$-coloring of $G - \int(H)$. Let $X=\varphi(h_{n})$. We have $\vert \varphi(v_1) \cap \varphi(v_2) \vert \leq e$, thus there exists $I \subset \varphi(v_2) - \varphi(v_1)$ such that $\vert I \vert =b-e$.  If $\vert X  \cap (\{1,\ldots,a\} - I - \varphi(v_1)) \vert \leq e$, then there exists $Y \subset \{1,\ldots,a\} - I - \varphi(v_1) - X$ such that $\vert Y \vert =e$. Otherwise, $\vert X  \cap (\{1,\ldots,a\} - I - \varphi(v_1)) \vert > e$, and there exists $Y' \subset (\{1,\ldots,a\} - I - \varphi(v_1)) \cap X$ such that $\vert Y' \vert =e-  \vert \{1,\ldots,a\} - I - \varphi(v_1) - X\vert$. We then choose $Y=(\{1,\ldots,a\} - I - \varphi(v_1) - X) \cup Y'$. By construction, $\vert X \cap( I \cup Y) \vert \leq b-e \leq e \frac{\even(\frac{2b}{e}) - 2}{2}$, hence by Theorem \ref{theoremalgorithmcoloration}, there exists a $(2b+e,b)$-coloring $\varphi'$ of $H$ such that $\varphi'(v_{1})=\varphi(v_1)$, $\varphi'(v_{2})=\varphi(v_2)$, $\varphi'(h_{n})=X$ and $\varphi'(h_{0})=I\cup Y$. 
\end{proof}

\begin{proposition}
\label{propositionStar(2b-k,k,k)}
For any graph $G$ and any integer $k$ with $2 \leq k \leq b$, $S(2b-k,k,k)$ is a smallest $(2b+1,b)$-reducible S-handle in $G$. 
\end{proposition}

\begin{proof}Let $n=2b-k$.

The minimality follows from the fact that the handles $S(2b-k,k,k-1)$ and $S(2b-k-1,k,k)$ are both not reducible since in both of them there is a path of length $2b-1$ between $v_2$ and $h_n$, hence the color sets of these two vertices must share at least $b-1$ common colors.

Let $G$ be a graph containing a handle $H=S(2b-k,k,k)$, with $k$ an integer such that $2 \leq k \leq b$, and let $\varphi$ be a $(2b+1,b)$-coloring of $G - \int(H)$.
%, we note $v_n$, $v'_{n'}$ and $v''_{n''}$ the other end-vertices of their handles, and 
Let $\varphi(v_1)=C_0$ and $\varphi(v_2)=C_{2k}$. Then $C_0,C_{2k}$ are $2k$-compatible. We are going to construct a good set $C_k$ such that $C_k,C_0$ and $C_k,C_{2k}$ are both $k$-compatible and $C_k, \varphi(h_n)$ are $(2b-k)$-compatible. Then, Theorem \ref{theoremalgorithmcoloration} will assert that there exists a  $(2b+1,b)$-coloring $\varphi'$ of $H$, such that $\varphi'(h_n)=\varphi(h_n)$,  $\varphi'(v_{1})=\varphi(v_{1})=C_0$ and $\varphi'(v_{2})=\varphi(v_{2})=C_{2k}$.
\begin{description}
 \item[Case 1 :] $k=2m$ is even.  By Lemma \ref{lemmaP4nnmilieu}  there exists $X$ such that $\vert X \vert = b+2m$, then $\vert X \cap \varphi(v_n)  \vert \geq b+2m + b - a=2m-1 \geq m$, so there exists $X' \subset  X \cap \varphi(h_n)$ such that $ \vert X' \vert =m $ and $X' \subset C_{k}$. Therefore $\vert C_k \cap \varphi(h_n) \vert \geq  m=b- \frac{2b-k}{2}$.
\item[Case 2 :] $k=2m+1$ is odd. By Lemma \ref{lemmaP4n+2nmilieu}  there exists $X$ such that $\vert X \vert = b+2m+1$, then $\vert X \cap \varphi(h_n)  \vert \geq b+2m+1 + b - a=2m \geq m+1$, so there exists $X' \subset  X \cap \varphi(h_n)$ such that $ \vert X' \vert =m+1 $ and $X' \cap C_{k}= \emptyset$. Therefore $\vert C_k \cap \varphi(h_n) \vert \leq b-(m+1)=\frac{(2b-k) -1}{2}$.
\end{description}
\end{proof}

\subsection{H-handle reductions}

\begin{proposition}
\label{proposition2handle}
For any $b,e$ with $b\ge e$ and any graph $G$, $H(1,2, \even( \frac{2b}{e} )-2,2,1 )$ is a smallest $(2b+e,b)$-reducible H-handle in $G$.
\end{proposition}

\begin{proof}Let $n=\even(\frac{2b}{e} ) - 2$.

For the minimality, it can be observed that $H(1,1, n,2,1 )$ and $H(1,2, n-1,2,1 )$ are both not reducible since the first one is in fact an S-handle $S(n,2,1 )$ and is not reducible by Proposition~\ref{proposition1handle} and the second one has a path of length $1+n-1+1=n+1< \even(\frac{2b}{e} )$ between two of its extremities and thus is not reducible.

Now let $G$ be a graph containing a handle $H=H(1,2, n,2,1)$ and let $\varphi$ be a $(2b+e,b)$-coloring of $G - \int(H)$. Let $v'_2$ be the common neighbor of $h_0$ and $v_2$ and $v'_3$ be the common neighbor of $h_n$ and $v_3$.
By hypothesis, $|\varphi(v_1) \cap \varphi(v_2)|\le e$ and $|\varphi(v_3) \cap \varphi(v_4)|\le e$, hence $|\varphi(v_2) - \varphi(v_1)|\ge b-e$ and $|\varphi(v_3) - \varphi(v_4)|\ge b-e$. Let $X\subset \varphi(v_2) - \varphi(v_1)$ and $X'\subset \varphi(v_3) - \varphi(v_4)$ such that $|X|=|X'|=b-e$. We are going to show that sets $Y\subset \{1,\ldots, a\}-\varphi(v_1)$ and $Y'\subset \{1,\ldots, a\}-\varphi(v_4)$ can be chosen in such a way that $C_0=X\cup Y$ and $C_n=X'\cup Y'$ are $n$-compatible, i.e., $|C_0\cap C_n|\ge e$. Then, Theorem \ref{theoremalgorithmcoloration} will assert that there exists a  $(2b+e,b)$-coloring $\varphi'$ of $H$, such that $\varphi'(h_0)=C_0$,  $\varphi'(h_n)=C_n$. % and the other paths of $H$ can be $(2b+e,b)$-colored by Theorem  \ref{theoremalgorithmcoloration}.% $\varphi'(v)=\varphi(v)$ for the other vertices.

Let $L(h_n)=\{1,\ldots,a\} - \varphi(v_4)$ and $L(h_0)=\{1,\ldots,a\} - \varphi(v_1)$. If $\vert X \cap X' \vert \geq e$, then $\vert C_0 \cap C_n \vert \geq e$. Hence assume $\vert X \cap X' \vert = z < e$.
 If $ \vert ( L(h_n) - X' ) \cap X \vert \geq e-z$, then we choose $Y'$ such that $ \vert Y' \cap X \vert \geq e-z$, therefore $\vert C_0 \cap C_n \vert \geq e$. Otherwise $\vert ( L(h_n) - X' ) \cap X \vert = t < e-z$. If $ \vert ( L(v_0) - X ) \cap X' \vert \geq e-z-t$, then we choose $Y$ such that $\vert Y  \cap X' \vert \geq e-z-t$, therefore $\vert C_0 \cap C_n \vert \geq e$. Otherwise $\vert ( L(v_0) - X ) \cap X' \vert = t' < e-z-t$. We note $z'= \vert (L(v_0)-X) \cap (L(h_n)-X') \vert $. Since $a \geq \vert L(v_0) \cup L(h_n) \vert =2b + 2e -z-z'-t-t'$, we have $z+z'+t+t' \geq e$. We choose $Y_0 \subset (L(v_0)-X) \cap (L(h_n)-X')$ such that $\vert Y_0 \vert = e-z-t-t'$. We choose $Y,Y'$ such that $Y_0 \cup \big( ( L(h_n) - X' ) \cap X \big) \subset Y'$ and $Y_0 \cup \big( ( L(v_0) - X ) \cap X' \big) \subset Y$, therefore $\vert C_0 \cap C_n \vert \geq e$.
 
\end{proof}

In order to shorten some proofs for H-reducibility, we will use the two following lemmas.
\begin{lemma}
\label{lemmapassagedroite}
Let $n \geq 1$ and $G$ be a graph. If $H(n_{1},n_{2},n,n_{3},n_{4})$ is $(a,b)$-reducible in $G$ and  $H(n_{1},n_{2},n-1,n_{3}+1,n_{4}+1)$ is $(a,b)$-reducible in $G$ when the color-sets of $v_{3}, v_{4}$ are $(n_3 + n_4 +2)$-exactly-compatible, then $H(n_{1},n_{2},n-1,n_{3}+1,n_{4}+1)$ is $(a,b)$-reducible in $G$. 
\end{lemma}

\begin{proof}
Let $n \geq 1$. Let  $H$ be a handle  $H(n_{1},n_{2},n-1,n_{3}+1,n_{4}+1)$ in $G$ and $\varphi$ be a $(a,b)$-coloring of $G - \int(H)$. Assume $H(n_{1},n_{2},n,n_{3},n_{4})$ is $(a,b)$-reducible in $G$, and  $H(n_{1},n_{2},n-1,n_{3}+1,n_{4}+1)$ is $(a,b)$-reducible in $G$ for $\varphi(v_{3}), \varphi(v_{4})$ being $(n_3 + n_4 +2)$-exactly-compatible. If $\varphi(v_{3}), \varphi(v_{4})$ are $(n_3 + n_4 +2)$-exactly-compatible then it is true by hypothesis. Otherwise $\varphi(v_{3}), \varphi(v_{4})$ are $(n_3 + n_4)$-compatible. Then, by hypothesis, there exists a $(a,b)$-coloring $\varphi'$ of $H'=H(n_{1},n_{2},n,n_{3},n_{4})$ with $\varphi'(v'_i)=\varphi(v_i)$, where $v'_i$ is the end-vertex of the path of length $n_i$ of $H'$, $i=1,\ldots,4$. Now we complete the coloring of $G$: 
to each vertex of the paths of $H$ of length $n_1, n_2$, and $n-1$ starting at $h_0$ we associate the color of the corresponding vertex in $H'$; for the two neighbors of $h_{n-1}$ along the paths towards $v_3$ and $v_4$ we assign the color $\varphi'(h'_n)$; for the remaining subpaths of lengths $n_3$ and $n_4$ of $H$, we use the colors of the corresponding paths of the same lengths in $H'$.
We obtain a $(a,b)$-coloring of $H$ that completes the $(a,b)$-coloring $\varphi$ of $G$, hence $H$ is $(a,b)$-reducible in $G$.
%We note $v'$ the same end-vertex of $S(n-1,n_3 +1, n_4 +1)$. We construct $\varphi$ : if $v \in Star(n_{1},n_{2},n-1)$ then $\varphi(v)=C'(v)$; if $v \in H_{n_3 +1}  -v'$, we note $v'' \in H_{n_3}$ the vertex with the same distance : $d(v,v_{n_3+1})=d(v'',v_{n_3})$, then $\varphi(v)=C'(v'')$; else  $v \in H_{n_4 +1}  -v'$, we note $v'' \in H_{n_4}$ the vertex with the same distance : $d(v,v_{n_4+1})=d(v'',v_{n_4})$, then $\varphi(v)=C'(v'')$. Therefore $H(n_{1},n_{2},n-1,n_{3}+1,n_{4}+1)$ can be reduced.
\end{proof}

By symmetry and adding the new condition $n_1 < n_4$ in order to fulfill our convention, we have this lemma: 

\begin{lemma}
\label{lemmapassagegauche}
Let $n \geq 1$ and $n_1 < n_4$. If $H(n_{1},n_{2},n,n_{3},n_{4})$ is $(a,b)$-reducible in $G$ and $H(n_{1} + 1 ,n_{2} + 1 ,n-1,n_{3},n_{4})$ is $(a,b)$-reducible in $G$ when the color-sets of $v_{1}, v_{2}$ are $(n_1 + n_2 +2)$-exactly-compatible, then $H(n_{1} + 1 ,n_{2} + 1 ,n-1,n_{3},n_{4})$ is $(a,b)$-reducible in $G$. 
\end{lemma}

\bigskip

\begin{proposition}
\label{propositionTypeH(2,2,2b-3,2,2)}
For any graph $G$, the H-handles $H(2,2,2b-3,2,2)$,  $H(1,2,2b-3,3,2)$ and $H(1,4,2b-3,2,2)$ are the smallest $(2b+1,b)$-reducible H-handles $H(n_1,n_2,n,n_3,n_4)$ with $n=2b-3$ in $G$.
\end{proposition}

\begin{proof}
For proving minimality, we have the necessary condition $n_{1}  + n_{4} \geq 2b-(2b-3)=3$. Moreover, we have $n_4 < 3$. If $n_1=1$ then $n_3 \geq n_4=2$. Therefore $H(1,2,2b-3,2,2)$ could be the smallest $(2b+1,b)$-reducible handle. However there exists a counter example even for $H(1,3,2b-3,2,2)$: $C(v_{1})=\{1,\ldots,b\}$, $C(v_{2})=\{1,\ldots,b-2\} \cup \{2b,2b+1\}$, $C(v_{4})=\{b+1,\ldots,2b-1\} \cup \{2b+1\}$, $C(v_{3})=\{b+1,\ldots,2b\}$. Therefore, there remain only three configurations to be tested:  $H(1,2,2b-3,3,2)$, $H(1,4,2b-3,2,2)$, and $H(2,2,2b-3,2,2)$.
We now consider each of these three H-handles in turn and prove that it is $(2b+1,b)$-reducible. Let $n=2b-3$.

\paragraph{$H=H(2,2,2b-3,2,2)$.} Let $\varphi$ be a $(2b+1,b)$-coloring of $G-\int(H)$.
 We are going to show that there exist two $(2b-3)$-compatible sets $C$ and $C'$ that can be given to vertices $h_0$ and $h_n$. By Lemma \ref{lemmaP4} there exist four sets $X'_i, i\in \{1,2,3,4\}$ for the vertex $h_n$ with $\vert X'_1 \cap X'_4 \vert =b- 2$. Let $I=\varphi(v_1)\cap \varphi(v_2)$ and $I'=\varphi(v_3)\cap \varphi(v_4)$.
By hypothesis, both $\varphi(v_1)$ and $\varphi(v_2)$ and $\varphi(v_3)$ and $\varphi(v_4)$ are $4$-compatible. Hence the size of both $I$ and $I'$ is between $b-2$ and $b$. We then consider three cases depending on the values of $|I|$ and $|I'|$.
\begin{description}
\item[Case 1.] $|I|=b$ or $|I'|=b$. Assume without loss of generality that $|I|=b$.  Then either there exists $i\in\{1,4\}$ such that $|I\cap X'_i|\le b-2$ and thus we can set $C=I$ and $C'=X'_i$ or we have $|I\cap X'_i|\ge  b-1$ for all $i\in\{1,2,3,4\}$. In this case, we can set $C=X'_1$ and $C'=X'_4$.

\item[Case 2.] $|I|=b-1$ or $|I'|=b-1$. Assume without loss of generality that $|I|=b-1$. Since $|X_1\cap X_4|=b-2$, then there exists $i\in\{1,4\}$ such that $|I\cap X'_i|\le b-2$ and thus we can set $C=I\cup \{x\}$ and $C'=X'_i$, with $x\in \{1,\ldots, a\}-I - X'_i$.

\item[Case 3.] $|I|=|I'|=b-2$. If there exist $x\in \varphi(v_1)-I$, $y\in \varphi(v_2)-I$ and $i\in\{1,4\}$ such that $|(I\cup\{x,y\})\cap X'_i|\le b-2$ then we can set $C=I\cup\{x,y\}$ and $C'=X'_i$. Otherwise $|(I\cup\{x,y\})\cap X'_i|=b-1$, then $X'_1=\varphi(v_1)$ and $X'_4=\varphi(v_2)$ then $\varphi(v_1)=I\cup\{x,x'\}, \varphi(v_2)=I\cup\{y,y'\}, \varphi(v_3)=I\cup\{x,y\}$ and $\varphi(v_4)=I\cup\{x',y'\}$. In this case we can choose $C=I\cup \{x,y'\}$ and $C'=I\cup \{x',y\}$.
\end{description}

\paragraph{$H=H(1,2,2b-3,3,2)$.} Let $\varphi$ be a $(2b+1,b)$-coloring of $G-\int(H)$.
%, we note $C_0=\varphi(v_{n_1})$, $C_3=\varphi(v_{n_2})$, $C'_0=\varphi(v_{n_4})$ and $C'_5=\varphi(v_{n_3})$, then $\varphi(v_1)$ satisfies $(3,C_3)$, and $C'_0$ satisfies $(5,C'_5)$. Recall that $v,v'$ are the end-vertices of $H_n$. 
If $\vert \varphi(h_0)\cap \varphi(h_n) \vert \leq b-2$, then $\varphi(h_0)$ and $\varphi(h_n)$ are $(2b-3)$-compatible, therefore by Theorem \ref{theoremalgorithmcoloration} we conclude. Else $ \vert \varphi(h_0)\cap \varphi(h_n) \vert > b-2$. By Lemma \ref{lemmaP5} there exists $I,Y$ and $Z$ such that $\vert I \vert = b-2$ and $\vert Y \vert = \vert Z \vert = 2$. Let $Y=Y_1 \cup Y_2$ and $Z=Z_1 \cup Z_2$, with $|Y_i|=|Z_i|=1$, $i=1,2$. If there exists $i,j \in \{1,2\}$ such that $\vert \varphi(h_0)\cap ( I \cup Y_i \cup Z_j ) \vert \leq b-2$, then we conclude. Otherwise for any $i,j \in \{1,2\}$, $\vert \varphi(h_0)\cap ( I \cup Y_i \cup Z_j ) \vert >  b-2$, then without loss of generality, $\varphi(h_0)=I \cup Y$. But Lemma \ref{lemmaP2n+1un} allows to construct another coloring $\varphi'$ with $\vert \varphi'(h_0)\cap \varphi(h_0)\vert = b-1 $ and thus $I \cup Z \not= \varphi'(h_0)\not= I \cup Y$. 
Hence there exists $i,j \in \{1,2\}$ such that $\vert \varphi'(h_0)\cap ( I \cup Y_i \cup Z_j ) \vert \leq b-2$, allowing to conclude.

\paragraph{$H=H(1,4,2b-3,2,2)$.} Let $\varphi$ be a $(2b+1,b)$-coloring of $G-\int(H)$. Then $\varphi(v_1)$ and $\varphi(v_2)$ are $5$-compatible and $\varphi(v_3)$ and $\varphi(v_4)$ are $4$-compatible. If $\vert \varphi(h_0)\cap \varphi(h_n) \vert \leq b-2$, then $\varphi(h_0)$ and $\varphi(h_n)$ are $(2b-3)$-compatible, therefore by Theorem \ref{theoremalgorithmcoloration} we conclude. Else $ \vert \varphi(h_0)\cap \varphi(h_n) \vert > b-2$. By Lemma \ref{lemmaP4} we have  4 good sets $X'_1$,  $X'_2$,  $X'_3$  and  $X'_4$ for the vertex $h_n$, and by Lemma \ref{lemmaP2n+1un}, there exists a set $X$ such that $|X|=b-2$ and a set $Y=\{y_1,y_2,y_3\}$ such that the three good sets $C_1=X \cup \{y_1,y_2\}$, $C_2=X \cup \{y_1,y_3\}$ and $C_3=X \cup \{y_2,y_3\}$ can be given to vertex $h_0$. If there exists $i,j \in \{1,2,3,4\}$ with $i \ne 4$, such that  $\vert X'_j \cap C_i \vert \leq b-2$, then we conclude. Else $\forall i,j \in \{1,2,3,4\}$  with $i \not= 4$, we have $ b \geq \vert X'_j \cap C_i \vert > b-2$, and since $\vert X'_1 \cap X'_4 \vert =b- 2$, then for all $j \in \{1,4\}$ and for all $i \in \{1,2,3\}$ we have the formula $\vert X'_j \cap C_i \vert =  b-1$, so  $X'_1 \cap X'_4 \subset C_i \subset X'_1 \cup X'_4$. Without loss of generality, $y_1 \in X'_1$, thus by the formula with $i=1,j=1$ we have $y_2 \notin X'_1$. Therefore by the formula with $i=3,j=1$ we have $y_3 \in X'_1$, a contradiction with the formula with $i=2,j=1$. 
\end{proof}

%\bigskip

%%%%%%%%%%%%%%%%%%%%%%%%%%%%%%%%%%%%%%%%%%%%%%%%%%%%%%%%%%%%%%%%
%%%%%%%%%%%%%%%%%%%%%%%%%%%%%%%%%%%%%%%%%%%%%%%%%%%%%%%%%%%%%%%%

%\bigskip

\begin{proposition}
\label{propositionTypeH(2,3,2b-4,3,2)}
For any graph $G$, the H-handles $H(1,2,2b-4,4,3)$, and $H(1,4,2b-4,3,3)$ are the smallest $(2b+1,b)$-reducible H-handles $H(n_1,n_2,n,n_3,n_4)$ with $n_1=1$ and $n=2b-4$ in $G$.
\end{proposition}

\begin{proof}
For proving minimality, we have the necessary condition $n_{1}  + n_{4} \geq 2b-(2b-4)=4$. %Moreover we have $n_1 \leq n_4 < 4$. 
%Hence, if $n_1\ge 2$, then $H(2,2,2b-4,2,2)$ could be the smallest $(2b+1,b)$-reducible H-handle, however there exists a counter example: $C(v_{1})=\{1,\ldots,b\}$,  $C(v_{2})=\{3,\ldots,b+2\}$,  $C(v_{4})=\{1,b+3,\ldots,2b+1\} $ and $C(v_{3})=\{ 2,b+3,\ldots,2b+1\}$. 
Hence, as $n_1=1$, then $n_3 \geq n_4=3$. If $n_2\le 3$ then we can construct a simple counter example for $H(1,3,2b-4,3,3)$: $C(v_{1})=\{1,\ldots,b\}$, $C(v_{2})=\{3,\ldots,b+2\}$, $C(v_{4})=\{b+3,\ldots,2b+1\} \cup \{b+1\}$, $C(v_{3})=\{b+2,\ldots,2b+1\}$. 
Therefore there remain two possibilities in total: $H(1,2,2b-4,4,3)$ and $H(1,4,2b-4,3,3)$. 
We consider each of these two H-handles in turn and prove that it is $(2b+1,b)$-reducible in $G$.

%\paragraph{$H=H(2,2,2b-4,3,2)$.} We prove it for $b=4$ but the proof can be extended to any $b$ using the same arguments. Let $\varphi$ be a $(9,4)$-coloring of $G-\int(H)$ among the colors $C=\{1,\ldots, 9\}$. We note $C_1=\varphi(v_{1})$, $C_2=\varphi(v_{2})$, $C_3=\varphi(v_{3})$ and $C_4=\varphi(v_{4})$, then $C_1$, $C_2$ are 4-compatible and $C_3$,  $C_4$ are 5-compatible. Let $I=C_1\cap C_2$. We consider three cases depending on $|I|$:
%\textbf{Case 1:} $|I|=4$. Then, we must have $|\varphi(h_0)\cap I|\ge 3$.
%If $|C_4\cap I| \ge 1$, then let $y\in C_4\cap I$ and let $x \in C_4\setminus (C_3\cup \{y\})$. Such $x$ exists since $ |C_3\cap C_4|\le 2$. Then assign colors $\{x,y\}$ completed by two colors from $I\setminus \{x,y\}$  to $h_0$, and colors $\{x,y\}$ completed by one color from $C_4\setminus (C_3\cup\{x,y\})$ and one color from $C\setminus le reste$ to $h_4$.

\paragraph{$H=H(1,2,2b-4,4,3)$.} Let $\varphi$ be a $(2b+1,b)$-coloring of $G-\int(H)$. We note $C_0=\varphi(v_{1})$, $C_3=\varphi(v_{2})$, $C'_0=\varphi(v_{4})$ and $C'_7=\varphi(v_{3})$, then $C_0,C_3$ are $3$-compatible, and $C'_0,C'_7$ are 7-compatible. By Proposition  \ref{propositionTypeH(2,2,2b-3,2,2)}, $H(1,2,2b-3,3,2)$ is $(2b+1,b)$-reducible, and by Lemma \ref{lemmapassagedroite}, there only remains to prove that $H$ is reducible in the case $C'_0,C'_7$ are $7$-exactly-compatible.
We are going to show that there exist two sets of $b$ colors $C_1$ and $C'_3$ that can be given to vertices $h_0$ and $h_n$, respectively.
We note $C'_0\cap C'_7=Y'_1 \cup Y'_2 \cup Y'_3  $, $I'=C'_7 - C'_0$, $I''=\{1,\ldots,2b+1\}-C'_7-C'_0=Z'_1 \cup Z'_2 \cup Z'_3 \cup Z'_4 $, $D'_1=I' \cup Y'_1 \cup Z'_1 \cup Z'_2 $, $D'_2=I' \cup Y'_2 \cup Z'_3 \cup Z'_4 $, $D'_3=I' \cup Y'_3 \cup Z'_1 \cup Z'_4 $ and $D'_4=I' \cup Y'_2 \cup Z'_2 \cup Z'_3$. 
By Lemma \ref{lemmaP2n+1un}  there exist two (distinct) good sets $D_1$ and $D_2$, for $C_1$. If there exists $i \in \{1,2\}$ and $j \in \{1,2,3,4\}$ such that $\vert D'_j \cap D_i \vert \geq 2$, then we set $C_1=D_i$ and  $C'_3=D'_j$.  Else  for any $i \in \{1,2\}$ and $j \in \{1,2,3,4\}$ we have $\vert D'_j \cap D_i \vert \leq 1$. However $\vert (D'_1 \cup D'_2 \cup D'_3) \cap D_i \vert  \geq b+4+b-a=3$ and $\vert (D'_1 \cup D'_4 \cup D'_3) \cap D_i \vert  \geq 3$. Thus  $\vert D'_j \cap D_i \vert = 1$  and $\vert (D'_1 \cup D'_2 \cup D'_3) \cap D_i \vert  =\vert (D'_1 \cup D'_4 \cup D'_3) \cap D_i \vert  =3$. Then for any $i=1,2$, $C'_0\cap C'_7 \subset D_i$ and $D_i \cap  (I' \cup I'')= \emptyset$, and therefore $D_i=C'_0$, a contradiction with $D_1 \not= D_2$.

\paragraph{$H=H(1,4,2b-4,3,3)$.} Let $\varphi$ be a $(2b+1,b)$-coloring of $G-\int(H)$, we note $C_0=\varphi(v_{1})$, $C_5=\varphi(v_{2})$, $C'_0=\varphi(v_{4})$ and $C'_6=\varphi(v_{n})$, then $C_0,C_5$ are 5-compatible and $C'_0,C'_6$ are 6-compatible. By Proposition  \ref{propositionTypeH(2,2,2b-3,2,2)}, $H(1,4,2b-3,2,2)$ is $(2b+1,b)$-reducible and, by Lemma \ref{lemmapassagedroite}, there only remains to prove the reduction for the case $C'_0,C'_6$ are 6-exactly-compatible. 
We are going to show that there exist two sets of $b$ colors $C_1$ and $C'_3$ that can be given to vertices $h_0$ and $h_n$, respectively.
We note $C'_0-C'_6=Y'_1 \cup Y'_2 \cup Y'_3  $, $C'_6-C'_0=Z'_1 \cup Z'_2 \cup Z'_3  $, $I'=\{1,\ldots,2b+1\}-C'_6-C'_0$, $D'_1=I' \cup Y'_1 \cup Z'_1 $, $D'_2=I' \cup Y'_2 \cup Z'_2 $ and $D'_3=I' \cup Y'_3 \cup Z'_3$. By Lemma \ref{lemmaP2n+1un}  there exist $I$ such that $\vert I \vert = b-2$, and $X=X_1 \cup X_2 \cup X_3$ such that $D_1=I \cup X_1 \cup X_2 $, $D_2=I \cup X_1 \cup X_3 $ and $D_3=I \cup X_2 \cup X_3 $. For any $i,j \in \{1,2,3\}$, if $\vert D'_j \cap D_i \vert \geq 2$, then we set $C_1=D_i$ and  $C'_3=D'_j$.  Else  for any $i,j \in \{1,2,3\}$  we have $\vert D'_j \cap D_i \vert \leq 1$. However $\vert (D'_1 \cup D'_2 \cup D'_3) \cap D_i  \vert  \geq b+3+b-a=2$. Thus $I \cap I'=\emptyset$ and if $X_1 \subset I'$ then $(I \cup X_2 \cup X_3) \cap (D'_1 \cup D'_2 \cup D'_3)=\emptyset $, a contradiction. Proceeding similarly for $X_2, X_3$, we are in the case that $X \cap I'=\emptyset$ and thus $\vert (I \cup X) \cap ((C'_0-C'_6) \cup (C'_6-C'_0)) \vert \geq b+1+6+b-2-a=4$. Hence, without lost of generality, $(C'_0-C'_6) \cup Z'_1 \subset I \cup X$, a contradiction with $\vert D'_1 \cap D_i \vert \leq 1$.
\end{proof}

\appendix
\label{app}

\section{Other reducible $(9,4)$-configurations}

%\vspace*{-1cm}

\begin{figure}[H]
\centering
\def\ax{0}\def\ay{25}
\def\bx{10}\def\by{25}
\def\cx{20}\def\cy{25}
\begin{tikzpicture}[scale=0.5]
\node at (\ax+2,\ay) [circle,draw=black,scale=0.8](t1){};
\node at (\ax+5,\ay) [circle,draw=black,scale=0.8](t2){};
\node at (\ax+.2,\ay+1.2) [circle,draw=black,fill=black,scale=0.5](h1){};
\node at (\ax+.2,\ay-1.2) [circle,draw=black,fill=black,scale=0.5](h2){};
%\node at (\ax+6.4,\ay+1) [circle,draw=black,fill=black,scale=0.5](h3){};
%\node at (\ax+6.4,\ay-1) [circle,draw=black,fill=black,scale=0.5](h4){};
\node at (\ax-1,\ay-1.8) [circle,draw=black,fill=black,scale=0.5](h21){};
\node at (\ax-1,\ay-.6) [circle,draw=black,fill=black,scale=0.5](h22){};
\draw  (t1) -- (h1) node[midway, above] {$b$};
\draw  (h2)-- (t1) node[midway, below] {$c$};
%\draw [thick, style=dashed] (h1) .. controls+(-0.6,-1) .. (h2) node[midway, right] {2};
\draw [ultra thick] (t1) -- (t2) node[midway, above] {$a$};
\draw  (h2) -- (h21) node[midway, below] {$d$};
\draw  (h2) -- (h22) node[midway, above] {$e$};
%\draw [thick, style=dashed]  (h2) .. controls+(-.5,-1) and +(.5,-1) .. (h4)  node[midway, below] {4};
\node at (\ax+2,\ay-3.4) (){$S'(a,b,c,d,e)$};

\node at (\bx+2,\by) [circle,draw=black,scale=0.8](t1){};
\node at (\bx+5,\by) [circle,draw=black,scale=0.8](t2){};
\node at (\bx+.2,\by+1.4) [circle,draw=black,fill=black,scale=0.5](h1){};
\node at (\bx+.2,\by-1.4) [circle,draw=black,fill=black,scale=0.5](h2){};
%\node at (\bx+6.4,\by+1) [circle,draw=black,fill=black,scale=0.5](h3){};
%\node at (\bx+6.4,\by-1) [circle,draw=black,fill=black,scale=0.5](h4){};
\node at (\bx-1,\by+.8) [circle,draw=black,fill=black,scale=0.5](h11){};
\node at (\bx-1,\by+2) [circle,draw=black,fill=black,scale=0.5](h12){};
\node at (\bx-1,\by-2) [circle,draw=black,fill=black,scale=0.5](h21){};
\node at (\bx-1,\by-.8) [circle,draw=black,fill=black,scale=0.5](h22){};
\draw  (t1) -- (h1) node[midway, above] {$b$};
\draw  (h2)-- (t1) node[midway, below] {$c$};
%\draw [thick, style=dashed] (h1) .. controls+(-0.6,-1) .. (h2) node[midway, right] {2};
\draw [ultra thick] (t1) -- (t2) node[midway, above] {$a$};
\draw  (h1) -- (h11) node[midway, below] {$d$};
\draw  (h1) -- (h12) node[midway, above] {$e$};
\draw  (h2) -- (h21) node[midway, below] {$f$};
\draw  (h2) -- (h22) node[midway, above] {$g$};
\node at (\bx+3,\by-3.4) (){$S''(a,b,c,d,e,f,g)$};

\node at (\cx+2,\cy) [circle,draw=black,scale=0.8](t1){};
\node at (\cx+5,\cy) [circle,draw=black,scale=0.8](t2){};
\node at (\cx+.2,\cy+1.4) [circle,draw=black,fill=black,scale=0.5](h1){};
\node at (\cx+.2,\cy-1.4) [circle,draw=black,fill=black,scale=0.5](h2){};
\node at (\cx+6.8,\cy+1.4) [circle,draw=black,fill=black,scale=0.5](h3){};
\node at (\cx+6.8,\cy-1.4) [circle,draw=black,fill=black,scale=0.5](h4){};
\node at (\cx+8,\cy+.8) [circle,draw=black,fill=black,scale=0.5](h11){};
\node at (\cx+8,\cy+2) [circle,draw=black,fill=black,scale=0.5](h12){};
\node at (\cx+8,\cy-2) [circle,draw=black,fill=black,scale=0.5](h21){};
\node at (\cx+8,\cy-.8) [circle,draw=black,fill=black,scale=0.5](h22){};
\draw  (t1) -- (h1) node[midway, above] {$a$};
\draw  (h2)-- (t1) node[midway, below] {$b$};
\draw  (t2) -- (h3) node[midway, above] {$d$};
\draw  (h4)-- (t2) node[midway, below] {$e$};
%\draw [thick, style=dashed] (h1) .. controls+(-0.6,-1) .. (h2) node[midway, right] {2};
\draw [ultra thick] (t1) -- (t2) node[midway, above] {$c$};
\draw  (h3) -- (h11) node[midway, below] {$f$};
\draw  (h3) -- (h12) node[midway, above] {$g$};
\draw  (h4) -- (h21) node[midway, below] {$h$};
\draw  (h4) -- (h22) node[midway, above] {$i$};
\node at (\cx+4,\cy-3.4) (){$H''(a,b,c,d,e,f,g,h,i)$};
\end{tikzpicture}

\medskip
%\begin{table}[ht]
%\centering
\begin{tabular}{ ccccc }
\hline
	$a$ & $b$ & $c$ & $d$ & e\\ \hline
	7 & 1 & 1 & 3 & 6\\
	7 & 1 & 1 & 4 & 5\\
	6 & 3 & 1 & 3 & 3\\
	6 & 3 & 1 & 2 & 4\\
	6 & 2 & 1 & 3 & 4\\
	6 & 2 & 1 & 2 & 5\\
	5 & 4 & 2 & 2 & 2\\
	5 & 4 & 1 & 3 & 3\\
	5 & 4 & 1 & 2 & 4\\
	5 & 3 & 2 & 3 & 2\\
	5 & 3 & 1 & 4 & 4\\
	4 & 4 & 3 & 2 & 2\\
	4 & 4 & 3 & 1 & 4\\
	4 & 4 & 2 & 2 & 3\\
	4 & 4 & 1 & 3 & 4\\%\hline
\end{tabular}\hspace*{1.2cm}
\begin{tabular}{ ccccccc }
\hline
	$a$ & $b$ & $c$ & $d$ & $e$ & $f$ & $g$\\ \hline%\hline
	7 & 1 & 1 & 3 & 4 & 3 & 4\\
	7 & 1 & 1 & 3 & 3 & 4 & 4\\
	7 & 1 & 1 & 3 & 3 & 3 & 5\\
	6 & 2 & 1 & 2 & 2 & 3 & 3\\
	6 & 2 & 1 & 3 & 3 & 2 & 3\\
	6 & 1 & 1 & 2 & 4 & 3 & 4\\
	6 & 1 & 1 & 3 & 3 & 4 & 4\\
	6 & 1 & 1 & 2 & 5 & 2 & 5\\	
	5 & 3 & 2 & 2 & 2 & 2 & 2\\
	5 & 2 & 2 & 2 & 3 & 2 & 3\\
	5 & 2 & 2 & 3 & 3 & 2 & 2\\
	5 & 2 & 2 & 2 & 4 & 2 & 2\\
	5 & 2 & 1 & 2 & 3 & 3 & 4\\
	5 & 2 & 1 & 3 & 3 & 3 & 3\\
	5 & 1 & 1 & 3 & 3 & 4 & 4\\
	5 & 1 & 1 & 3 & 4 & 3 & 4\\
	4 & 4 & 2 & 1 & 2 & 2 & 2\\
	4 & 4 & 1 & 2 & 2 & 3 & 3\\
	4 & 4 & 1 & 1 & 4 & 3 & 3\\
	4 & 3 & 3 & 1 & 2 & 2 & 2\\
	4 & 3 & 2 & 1 & 2 & 3 & 3\\
	4 & 3 & 1 & 2 & 2 & 3 & 4\\
	4 & 2 & 2 & 2 & 3 & 2 & 3\\
	3 & 3 & 3 & 2 & 3 & 2 & 3\\
	
\end{tabular}\hspace*{1.2cm}
\begin{tabular}{ ccccccccc }
\hline
	$a$ & $b$ & $c$ & $d$ & $e$ & $f$ & $g$ & $h$ & $i$\\ \hline
 	1 & 2 & 6 & 1 & 1 & 1 & 1 & 5 & 4\\
 	1 &  2 &  6 &  1 &  1 &  2 &  3 & 4 & 4\\
	2 &  2 &  5 &  1 &  1 &  2 &  3 &  4 & 4\\
	1 &  2 &  5 &  2 &  2 &  1 &  1 &  2 & 1\\
	1 &  2 &  5 &  4 &  1 &  2 &  2 &  2 & 2\\
	1 &  2 &  5 &  3 &  1 &  2 &  3 &  2 & 3\\
	1 &  2 &  5 &  2 &  1 &  1 &  3 &  4 & 4\\
	1 &  2 &  4 &  4 &  2 &  1 &  1 &  2 & 2\\
	1 &  2 &  4 &  4 &  2 &  2 &  2 &  2 & 1\\
	1 &  2 &  4 &  2 &  2 &  2 &  3 &  2 & 3\\
	2 &  2 &  4 &  4 &  1 &  2 &  2 &  3 & 2\\
	2 &  4 &  4 &  3 &  1 &  2 &  2 &  3 & 2\\
	2 &  4 &  4 &  2 &  1 &  2 &  2 &  3 & 3\\
	2 &  2 &  4 &  2 &  2 &  2 &  2 &  3 & 2\\
	3 &  3 &  2 &  3 &  3 &  1 &  2 &  2 & 2\\
\end{tabular}
\caption{\label{tbSp}$\mF''_{9,4}$: a new set of $(9,4)$-reducible configurations of type $S'(a,b,c,d,e)$ on the left, $S''(a,b,c,d,e,f,g)$ (middle), and $H''(a,b,c,d,e,f,g,h,i)$ (right).}
 %\caption{\label{sp} Extended S-handles and H-handles}
\end{figure}

\end{document}